\documentclass[a4paper,12pt,oneside,reqno]{amsart}
\usepackage{hyperref}
\usepackage[headinclude,DIV13]{typearea}
\areaset{15.1cm}{25.0cm}
\usepackage{amsfonts,amssymb,amsmath,amsthm,bbm}
\usepackage{cancel} 
\usepackage[latin1] {inputenc}
\usepackage{graphicx, psfrag}

\newtheorem{theorem}{Theorem}[section]
\newtheorem{lemma}[theorem]{Lemma}
\newtheorem{proposition}[theorem]{Proposition}

\theoremstyle{definition}

\theoremstyle{remark}
\newtheorem*{remark}{Remark}

\def\paragraph#1{\noindent \textbf{#1}}

\numberwithin{equation}{section}

\def\dist{\mathop{\rm dist}\nolimits}

\def\rf{\rfloor}

\def\lf{\lfloor}

\def\<{\langle}
\def\>{\rangle}
\def\a{\alpha}
\def\b{\beta}
\def\e{\epsilon}

\def\g{\gamma}

\def\l{\lambda}

\def\s{\sigma}
\def\t{\tau}
\def\th{\theta}

\def\o{\omega}

\def\O{\Omega}
\def\S{\Sigma}

\def\R{{\Bbb R}}  
\def\N{{\Bbb N}}  
\def\P{{\Bbb P}}  
\def\Q{{\Bbb Q}}  
\def\E{{\Bbb E}}  

\let\cal=\mathcal

\def\BB{{\cal B}}
\def\CC{{\cal C}}

\def\EE{{\cal E}}
\def\FF{{\cal F}}

\def\LL{{\cal L}}

\def\PP{{\cal P}}

\def\VV{{\cal V}}

\def\VV{{\cal V}}

 \def \b {{\beta}}
 \def \e {{\varepsilon}}
 \def \s {{\sigma}}

 \def \t {{\tau}}
 \def \th {{\theta}}
 
 \def \g {{\gamma}}
 \def \l {{\lambda}}
 
 \def \a {{\alpha}}
 \def \o {{\omega}}
 \def \O {{\Omega}}
 \def \th {{\theta}}

 \def \ba {\begin{array}}
 \def \ea {\end{array}}

 \newcommand{\be}{\begin{equation}}
 \newcommand{\ee}{\end{equation}}

\newcommand{\bea}{\begin{eqnarray}}
 \newcommand{\eea}{\end{eqnarray}}
\def\TH(#1){\label{#1}}\def\thv(#1){\ref{#1}}
\def\Eq(#1){\label{#1}}\def\eqv(#1){(\ref{#1})}

 \def \1{\mathbbm{1}}
\def\wt {\widetilde}
\def\wh{\widehat}

\newcommand{\limj}{\stackrel{J_1}{\Longrightarrow}} 
\newcommand{\limm}{\stackrel{M_1}{\Longrightarrow}} 

\begin{document}

 \title[Convergence to extremal processes]
{Convergence to extremal processes
 in random environments and extremal ageing in SK models}
\author[A. Bovier]{Anton Bovier}
 \address{A. Bovier\\Institut f\"ur Angewandte Mathematik\\
Rheinische Friedrich-Wilhelms-Universit\"at\\ Endenicher Allee 60\\ 53115 Bonn, Germany}
\email{bovier@uni-bonn.de}
\author[V. Gayrard]{V\'eronique Gayrard}
 \address{V. Gayrard\\CMI, LAPT, Universit\'e de Provence\\
39, rue F. Joliot Curie\\13453 Marseille cedex 13, France}
\email{veronique@gayrard.net}
\author[A. \v Svejda]{Ad\'ela \v Svejda}
 \address{A. \v Svejda\\Institut f\"ur Angewandte Mathematik\\
Rheinische Friedrich-Wilhelms-Universit\"at\\ Endenicher Allee 60\\ 53115 Bonn, Germany}
\email{asvejda@uni-bonn.de}

\subjclass[2000]{82C44,60K35,60G70} \keywords{ageing, spin glasses,
random environments, clock process, L\'evy processes, extremal processes
}
\date{\today}

 \begin{abstract}
This paper extends recent results on ageing in mean field spin glasses on 
short time scales, obtained by Ben Arous and G\"un \cite{BAGun11}
in law with respect to the environment, to results that hold almost surely, 
respectively in probability, with respect to the environment. 
It is based on the methods put
forward in \cite{G10a,G10b} and naturally complements \cite{BG10}.
 \end{abstract}

\thanks{ A.B. is partially supported through the 
German Research Foundation in  the SFB 611 and
the Hausdorff Center for Mathematics.
V.G. thanks the IAM, Bonn University, the Hausdorff Center, and the SFB 611 for kind hospitality. A.S. is supported by the German 
Research Foundation in the Bonn International Graduate School in Mathematics.}
 \maketitle

 \section{Introduction and main results}
 \label{S1}

Spin glasses have, for the last decades, presented some of the most interesting
challenges to probability theory. Even mean-field models have prompted a 
1000 page monograph \cite{Tala-new1,Tala-new2} by one of the most eminent 
probabilists of our time. Despite  these efforts and remarkable and 
unexpected progress, a full understanding of the
equilibrium problem, i.e. a full description of the asymptotic geometry of the
Gibbs measures, is still outstanding. In this situation it is somewhat 
surprising that certain properties of their dynamics have been prone to 
rigorous analysis, at least for some limited choices of the dynamics. The 
reason for this is that interesting aspects of the dynamics occur on 
time-scales that are far shorter than  those of equilibration, and 
experiments made with spin glasses usually test the behaviour of the probe 
on such time scales. Indeed, equilibration is expected to take so long as
to become inaccessible to real experiments. The physically interesting issue is 
thus that of \emph{ageing} \cite{Bou92,BD95}, 
a property of time-time correlation  functions that
characterizes the slow decay to equilibrium characteristic for  these systems. 

The mathematical analysis has revealed an universal mechanism behind this 
phenomenon: the convergence of the \emph{clock-process}, that relates the
physical time to the number of ``moves'' of the process, to an $\a$-stable
subordinator (increasing L\'evy process) under proper rescaling. The parameter 
$\a$ can be thought of as an \emph{effective temperature}, that depends both
on the \emph{physical temperature} and the \emph{time scale} considered. 
This has been proven for $p$-spin Sherrington-Kirkpatrick (SK) models
for time scales of the order $\exp(\b\g n)$ (where $n$ is the number of sites
 in
the system) with $0<\g< \min\bigl(\beta,\zeta (p)\bigr)$, where $\zeta(p)$ is
 an increasing 
function of $p$ such that $\zeta(3)>0$ and $\lim_{p\uparrow\infty}\zeta(p)=
2\ln 2$.  Such a result was obtained first in \cite{BBC08} \emph{in law}
with respect to the random environment, and was later extended in \cite{BG10}
to almost sure (resp. in probability, for $p=3,4$) results. The progress 
in the latter paper was possible to a fresh view on the convergence of 
clock processes, introduced and illustrated in two papers \cite{G10a,G10b}.
They view the clock process  as a sum of \emph{dependent} random variables with 
a \emph{random distribution}, and then employ convenient  convergence 
criteria, obtained by Durrett and Resnick \cite{DR78} a long time ago, 
to prove convergence. This is explained in more detail below.

The conditions on the admissible time scales in these results have two 
reasons.  First, it emerges that $\a=\g/\b$, so one of the conditions is
simply that $\a\in(0,1)$. The upper bound $\g<\zeta(p)$ ensures that there 
will be no strong  long-distance correlations, meaning that the systems has 
not had time to discover the full correlation structure of the random 
environment. This condition is thus the stricter the smaller $p$ is, since correlations become weaker as $p$ increases. 

A natural questions to ask is what happens on time-scales that are 
sub-exponential in the volume $n$? This question was first addressed in 
a recent paper by Ben Arous and G\"un \cite{BAGun11}. This situation would correspond formally  to $\a=0$, but $0$-stable subordinators do not exist, so 
some new phenomenon has to appear. Indeed, Ben Arous and G\"un showed that the limiting 
objects appearing here are the so-called \emph{extremal processes}.
In the theory of sums of heavy tailed random variables this idea goes back to 
Kasahara \cite{Kas86} who showed  that by applying non-linear transformations
to the sums of $\a_n$-stable r.v.'s with $\a_n\downarrow 0$, 
extremal processes arise as limit processes. This program was implemented for
clock processes by Ben Arous and G\"un using the approach of \cite{BBC08}
to handle the problems of dependence of the random variables involved. 
As a consequence, their results are again in law with respect to the random 
environment. An interesting aspect of this work was that, due to the very short
time scales considered,  the case $p=2$, i.e. the original SK model, is 
also covered, whereas this is not the case for exponential times scales.

In the present paper we show that by proceeding along the line of \cite{BG10},
one can extend the results of Ben Arous and G\"un to \emph{quenched} results, 
holding for given random environments almost surely (if $p>4$) resp. in 
probability (if $2\leq p\leq 4$).
In fact,  the result we present for the $SK$ models is an application  
of an abstract result we establish, and that can be applied
presumably to all models where ageing was analysed, on the approriate 
time scales. 

Before stating our results, we begin by a concise description of the class 
of models we consider.

\subsection{Markov jump processes in random environments}\label{S11a}
Let us describe the general setting of \emph{Markov jump processes} in random environments that we consider
here. Let $G_n(\VV_n, \LL_n)$ be a sequence of loop-free  graphs
with set of vertices $\VV_n$ and set of edges $\LL_n$. The \emph{random environment} is a family of positive random variables, 
$\t_n(x), x\in \VV_n$, defined on a common probability space $(\O,\FF, \P)$. Note that 
in the most interesting situations the $\t_n$'s are correlated random variables.

On $\VV_n$ we consider a discrete time Markov chain $J_n$ with initial distribution $\mu_n$, transition probabilities $p_n(x,y)$, and transition graph $G_n(\VV_n, \LL_n)$. The law of $J_n$ is a priori random on the probability space of the environment.
We assume that $J_n$ is reversible and admits a unique invariant measure $\pi_n$.

The process we are interested in, $X_n$, is defined as a time change of $J_n$. To this end we set
\be\label{1.0}
\lambda_n(x) \equiv C \pi_n(x)/\tau_n(x) ,  
\ee
where $C>0$ is a model dependent constant, and define the clock process
\be\label{1.1}
\wt S_n(k)=\sum_{i=0}^{k-1}\lambda^{-1}_n(J_n(i)) e_{n,i}, \quad k\in \N\ ,
\ee
where $\left\{e_{n,i} :\ i \in \N_0 , n \in \N\right\}$ is an i.i.d. array of mean $1$ exponential random variables, independent of $J_n$ and the random environment. 
The continuous time process $X_n$ is then given by
\be\label{1.2}
X_n(t)=J_n (k) ,\quad \hbox{\rm if}\,\, \wt S_n(k)\leq t<\wt S_n(k+1)\quad
\hbox{\rm for some}\,\, k \in \N,\ t > 0 \ .
\ee
One verifies readily that $X_n$ is a continuous time Markov jump process with infinitesimal generator 
\be\label{1.2a}
\lambda_n(x,y)\equiv \lambda_n(x)p_n(x,y),
\ee
and invariant measure that assigns to $x\in\VV_n$ the mass $\tau_n(x)$.

To fix notation we denote by $\FF^J$ and $\FF^X$ the $\s$-algebras generated by the
variables $J_n$ and $X_n$, respectively.
We  write $P_{\pi_n}$ for the law of the process $J_n$, conditional on $\FF$,
i.e. for fixed realizations of the random environment.
Likewise we  call $\PP_{\mu_n}$ the law of $X_n$ conditional on $\FF$.

In \cite{G10a, G10b}  and \cite{BG10}, the main aim was to find criteria when there are constants, $a_n,c_n$, satisfying $a_n, c_n \uparrow \infty$, as 
$n\rightarrow \infty$, and such that the process 
\be
S_n(t)\equiv c_n^{-1}\wt S_n(\lfloor a_nt\rfloor)=c_n^{-1}\sum_{i=0}^{\lfloor a_n t \rfloor-1}
\lambda^{-1}_n (J_n(i))e_{n,i} ,\quad t>0,
\Eq(1.3)
\ee
converges in a suitable sense to a stable subordinator. The constants $c_n$ are the time scale on which we observe the
continuous time Markov process $X_n$, while $a_n$ is the number
of steps the jump chain $J_n$ makes during that time.
In order to get convergence to an $\alpha$-stable subordinator, for $\alpha\in(0,1)$, one typically requires that the $\lambda^{-1}$'s observed
on the time scales $c_n$  have a regularly varying tail distribution with index $-\alpha$.
In this paper we
ask when there are constants, $a_n,c_n, \alpha_n$, satisfying $a_n, c_n \uparrow \infty$ and $\alpha_n\downarrow 0$ respectively, as 
$n\rightarrow \infty$, and such that the process $\left(S_n\right)^{\alpha_n}$ converges in a suitable sense to an  extremal process. 




\subsection{ Main Theorems} \label{S11}
We now state three theorems, beginning with an abstract one
that we next specialize to the setting of Section \ref{S11a}. 
Specifically, consider  a triangular array of positive random variables, $Z_{n,i}$, 
defined on a probability space $(\O,\FF,\PP)$. Let $\alpha_n$ and $a_n$ be sequences such that $\alpha_n \downarrow 0$ and $a_n\uparrow \infty$ as $n \rightarrow \infty$, respectively. Our first theorem gives conditions that ensure that the sequence of processes $\left(S_n\right)^{\alpha_n}$, where $S_n(0)=0$ and 
\be\label{1.3a}
S_n(t)\equiv \sum_{i=1}^{\lfloor a_nt\rfloor} Z_{n,i}, \quad t>0,
\ee 
converges to an extremal process. Recall that an extremal process, $M$, 
is a continuous time process whose finite-dimensional distributions
are given as follows: for any $k\in \N$,  $t_1,\ldots,t_k>0$, and 
 $x_1\leq \ldots\leq x_k\in \R$,
\be
P\left(M(t_1)\leq x_1,\ldots, M(t_k)\leq x_k\right)=F^{t_1}\left(x_1\right) F^{t_2-t_1}\left(x_2\right)\cdots F^{t_k-t_{k-1}}\left(x_k\right),\label{defextr}
\ee
where $F$ is a distribution function on $\R$. 
\begin{theorem}\TH(general)
Let $\nu$ be a sigma-finite measure
on $(\R_+,\BB(\R_+))$ such that $\nu(0,\infty)=\infty$. Assume that there exist sequences $a_n,\alpha_n$ such that 
for all continuity points $x$ of the distribution function of
$\nu$, for all  $t>0$, in $\PP$-probability,
\be\Eq(1.4)
\lim_{n\rightarrow \infty}\sum_{i=1}^{ \lf a_nt\rf}\PP\left(Z_{n,i}^{\alpha_n}>x|\FF_{n,i-1}\right) =t\nu(x,\infty),
\ee
and
\be\Eq(1.5)
\lim_{n\rightarrow\infty}\sum_{i=1}^{\lf a_nt\rf}\left[\PP\left(Z_{n,i}^{\alpha_n}>x|\FF_{n,i-1}\right)\right]^2 =0,
\ee
where $\FF_{n,i}$ denotes the $\s$-algebra generated by the random
variables $Z_{n,j}, j\leq i$.
If, moreover, for all $t>0$
\be\Eq(1.6)
\limsup_{n\rightarrow \infty}
\Biggl(\sum_{i=1}^{\lf a_n t\rf} \EE \1_{Z_{n,i} \leq \delta^{1/ \alpha_n} } \delta^{-1/\alpha_n} Z_{n,i}\Biggr)^{\alpha_n}<\infty,  \quad \forall \ \delta >0 \ , 
\ee
then, as $n \rightarrow \infty$,
\be\Eq(1.7)
\textstyle{\left(S_n\right)^{\alpha_n}\limj M_\nu,}
\ee
where $M_\nu$ is an extremal process with one-dimensional distribution function $F(x)=\exp(-\nu(x,\infty))$. Convergence holds weakly on the space $D([0,\infty))$ equipped with the Skorokhod $J_1$-topology. 
\end{theorem}

In the sequel we denote by $\limj$ weak convergence in $D([0,\infty))$ equipped with the Skorokhod $J_1$-topology.


In order to use Theorem \ref{general} in the Markov jump process setting of Section \ref{S11a}, we specify $Z_{n,i}$. In doing this we will be 
guided by the knowledge acquired in earlier works  \cite{G10a, G10b, BG10}: introducing a new scale $\theta_n$ we take $Z_{n,i}$ 
to be a block sum of length $\theta_n$, i.e.~we set
\be\label{1.8}
Z_{n,i}\equiv \sum_{j=(i-1)\theta_n+1}^{i \theta_n} c_n^{-1} \lambda^{-1}_n(J_n(j)) e_{n,j} \ .
\ee
The r\^ole of $\theta_n$ is to de-correlate the variables $Z_{n,i}$ under the law $\PP_{\mu_n}$. In models with uncorrelated environments and where
the probability of revisiting points is small, one may hope to take $\theta_n=1$. 
When the environment is correlated and the chain $J_n$ is rapidly mixing,
one may try to choose $\theta_n\ll a_n$ in such a way that, 
the variables $Z_{n,i}$ are close to independent.
These two situations were encountered in the random hopping dynamics of the Random Energy Model in \cite{G10b}, and the $p$-spin models in 
\cite{BG10} respectively.
Theorem \ref{thmblock} below specializes Theorem \ref{general} to these $Z_{n,i}$'s.

For $y\in\VV_n$ and $u>0$ let
\be
Q^{u}_n(y)\equiv\PP_{y}\Biggl(
\sum_{j=1}^{\theta_n}\lambda^{-1}_n(J_n(j))e_{n,j}>c_n u^{1/ \alpha_n}\Biggr) 
\Eq(1.9)
\ee
be the tail distribution of the blocked jumps of $X_n$, when $X_n$ starts in $y$. 
Furthermore, for $k_n(t)\equiv\left\lf{\lf a_n t\rf}/{\theta_n}\right\rf$, $t>0$, and $u>0$ define
\bea\Eq(1.10)
\nu_n^{J,t}(u,\infty) &\equiv& \sum_{i=1}^{k_n(t)} \sum_{y\in\VV_n}p_n(J_n(\theta_n i),y)Q^{u}_n(y)\ ,
\\
(\s_n^{J,t})^2(u,\infty)
&\equiv& \sum_{i=1}^{k_n(t)}\left[ \sum_{y\in\VV_n}p_n(J_n(\theta_n i),y)Q^{u}_n(y)\right]^2\ .
\Eq(1.11)
\eea

Using this notation, we rewrite Conditions \eqref{1.4}-\eqref{1.6}.
Note that $Q^{u}_n(y)$ is a random variable on the probability space $(\O, \FF, \P)$, and so are the quantities $\nu_n^{J,t}(u,\infty)$ and $\s_n^{J,t}(u,\infty)$.  The conditions below are stated for fixed realization of the random environment as well as for given sequences
 $a_n$, $c_n$, $\theta_n$, and $\alpha_n$ such that $a_n,c_n\uparrow \infty$, and $\alpha_n\downarrow 0$ as $n\rightarrow\infty$.\\
\noindent{\bf Condition (1)}
Let $\nu$ be a $\s$-finite measure on $(0,\infty)$ with $\nu(0,\infty)=\infty$ and such that for all $t>0$ and all $u>0$
\be
\lim_{n\rightarrow \infty}P_{\mu_n}\Bigl( \left| \nu_n^{J,t}(u,\infty)-t\nu(u,\infty) \right|>\e\Bigr)=0\ ,\quad\forall\e>0\ .
\Eq(1.12)
\ee

\noindent{\bf Condition (2)}  For all $u>0$ and all $t>0$,
\be
\lim_{n\rightarrow \infty}P_{\mu_n}\Bigl( (\s_n^{J,t})^2(u,\infty)>\e\Bigr)=0\ ,\quad\forall\e>0\ .
\Eq(1.13)
\ee

\noindent {\bf Condition (3)}  For all $t>0$ and all $\delta>0$
\be\Eq(1.14)
\limsup_{n\rightarrow \infty}
\Biggl(\sum_{i=1}^{\lf a_n t\rf}\EE_{\mu_n} \1_{\{\lambda^{-1}_n(J_n(i))e_{n,i}\leq \delta^{1/\alpha_n}c_n\}} 
(c_n \delta^{1/\alpha_n})^{-1}\lambda^{-1}_n(J_n(i))e_{n,i}\Biggr)^{\alpha_n}<\infty \ .
\ee

\noindent{\bf Condition (0)}  For all $v>0$,
\be
\lim_{n \rightarrow \infty}\sum_{x\in\VV_n}\mu_n(x)e^{-v^{1/ \alpha_n}c_n \lambda_n(x)}=0\ .
\Eq(1.15)
\ee
For $t>0$ set
\be \label{1.16}
\left(S_n^b (t)\right)^{\alpha_n} \equiv \Biggl( \sum_{i=1}^{k_n(t)}
\Biggl(\sum_{j=\th_n(i-1)+1}^{\theta_n i}c_n^{-1}\lambda^{-1}_n(J_n(j))e_{n,j}\Biggr) + c_n^{-1}\lambda^{-1}_n(J_n(0)) e_{n,0} \Biggr)^{\alpha_n}.
\ee
\begin{theorem}\label{thmblock}
If for a given initial distribution $\mu_n$ and given sequences $a_n, c_n, \theta_n$, and $\alpha_n$, Conditions (0)-(3) are satisfied $\P$-a.s., respectively in $\P$-probability, then
\be \label{1.16a}
\left(S_n^b\right)^{\alpha_n}  \limj M_{\nu}\ ,
\ee
where convergence holds $\P$-a.s., respectively in $\P$-probability.
\end{theorem}

\begin{remark}
Theorem \ref{thmblock} tells us that the blocked clock process $(S_n^b)^{\alpha_n}$converges to $M_{\nu}$ weakly in $D([0,\infty))$ equipped with the Skorokhod $J_1$-topology. This implies that the clock process $(S_n)^{\alpha_n}$ converges to the same limit in the weaker $M_1$-topology (see \cite{BG10} for further discussion).
\end{remark}
\begin{remark}
The extra Condition (0) serves to guarantee that the last term in  \eqref{1.16} is asymptotically negligible.
\end{remark}
Finally, following \cite{BG10}, we specialize Conditions (1)-(3) under the assumption that the chain $J_n$ obeys a mixing condition (see Condition (2-1) below).  Conditions (1)-(2) of Theorem \ref{thmblock} are then reduced to laws of large numbers for the random variables 
 $Q^{u}_n(y)$.
Again we state these conditions for fixed realization of the random environment and given sequences
$a_n$, $c_n$, $\theta_n$, and $\alpha_n$.\\
\noindent{\bf Condition (1-1)} Let $J_n$ be a periodic Markov chain
with period $q$.
There exists a positive decreasing sequence $\rho_n$, satisfying
$\rho_n\downarrow 0$ as $n\rightarrow\infty$, such that, for all pairs $x,y\in\VV_n$, and all $i\geq 0$,
\be
\sum_{k=0}^{q-1}P_{\pi_n}\left(J_n(i+\theta_n+k)=y,J_n(0)=x\right)\leq (1+\rho_n)\pi_n(x)\pi_n(y)\ .
\Eq(1.17)
\ee

\noindent{\bf Condition (2-1)}
There exists a $\s$-finite measure $\nu$ with $\nu(0,\infty)=\infty$ and such that
\be\Eq(1.18)
\nu_n^{t}(u,\infty)\equiv k_n(t)\sum_{x\in\VV_n}\pi_n(x)Q^{u}_n(x)
\rightarrow t\nu(u,\infty)\ ,
\ee
and
\be\Eq(1.19)
(\s_n^{t})^2(u,\infty)\equiv k_n(t)\sum_{x\in\VV_n}\sum_{x'\in\VV_n}\pi_n(x)p_n^{(2)}(x,x')Q^{u}_n(x)Q^{u}_n(x')\rightarrow 0\,,
\ee
where 
$p_n^{(2)}(x,x')=\sum_{y\in\VV_n}p_n(x,y)p_n(y,x')$ are the 2-step transition probabilities.\\[2mm]

\noindent {\bf Condition (3-1)}
 For all $t>0$ and $\delta >0$
\be\Eq(1.20)
\limsup_{n\rightarrow \infty}
\left( \lf a_n t\rf \EE_{\pi_n} \1_{\{\lambda^{-1}_n(J_n(1))e_{n,1}\leq c_n \delta^{1/ \alpha_n} \}} c_n^{-1}\delta^{-1/ \alpha_n}\lambda^{-1}_n(J_n(1))e_{n,1}\right)^{\alpha_n} < \infty \ . 
\ee

\begin{theorem}\label{thmblock1}
Let $\mu_n=\pi_n$. If for given  sequences $a_n, c_n$, $\theta_n\ll a_n$, and $\alpha_n$, Conditions (1-1)-(3-1) and (0) are satisfied $\P$-a.s., respectively in $\P$-probability, then $(S_n^b)^{\alpha_n} \limj M_{\nu}$, $\P$-a.s., respectively in $\P$-probability.
\end{theorem}


\subsection{Application to the $p$-spin SK model.} \label{S12} In this section we illustrate the power of  
Theorem \ref{thmblock1} by applying it to the $p$-spin SK models, including the SK model itself, i.e. $p\geq 2$.
The underlying graph $\VV_n$ is the hypercube $\S_n=\{-1,1\}^n$.
The Hamiltonian of the $p$-spin SK model is a Gaussian process, $H_n$, on $\S_n$ with zero mean and covariance 
\be\Eq(1.21)
\E H_n(x)H_n(x') =nR_n(x,x')^p,
\ee
where $R_n(x,x')\equiv 1- \frac{2\dist(x,x')}{n}$ and $\dist(\cdot,\cdot)$ is the graph distance on $\S_n$,
\be\Eq(1.22)
\dist(x,x')\equiv\frac 12 \sum_{i=1}^n |x_i-x'_i|.
\ee
The random environment, $\t_n(x)$, is defined in terms of $H_n$ through
\be\Eq(1.23)
\t_n(x)\equiv \exp(\b H_n(x)),
\ee
where $\b>0$ is the inverse temperature. The Markov chain, $J_n$, is
chosen as the simple random walk on $\S_n$, i.e.
\be\Eq(1.24)
p_n(x,x')= \begin{cases} \frac 1n, &\hbox{\rm if}\, \dist(x,x')=1,\\
0,&\hbox{\rm else}.
\end{cases}
\ee
This chain has unique invariant measure $\pi_n(x)=2^{-n}$. 
Finally, choosing $C=2^n$ in \eqref{1.0}, the mean holding times, $\l^{-1}_n(x)$,
reduce to $\l^{-1}_n(x)= \t_n(x)$.
This defines the so-called \emph{random hopping dynamics}.

In the theorem below the inverse temperature $\beta$ is to be chosen as a sequence $(\beta_n)_{n\in \N}$ that either  diverges or converges to a strictly positive limit. 
\begin{theorem}
  \label{p:main}
Let $\nu$ be given by $\nu(u,\infty)\equiv K_p u^{-1}$ for $u \in (0,\infty)$ and $K_p= 2 p$. Let $\gamma_n, \beta_n$ be such that $\gamma_n=n^{-c}$ for $c \in \left(0,\frac 12\right)$, $\beta_n\geq \beta_0$ for some $\beta_0>0$, and $\gamma_n \beta_n \leq O(1)$. Set $\alpha_n\equiv \gamma_n/\beta_n$.
\noindent Let $\theta_n= 3n^2$ be the block length and define the jump scales $a_n$ and time scales $c_n$ via
\bea\label{1.25}
a_n &\equiv& \sqrt{2\pi n}\ \gamma_n^{-1}\ e^{\frac 12 \gamma_n^2 n},\\
c_n &\equiv& e^{\gamma_n \beta_n n}. \ \label{1.26}
\eea
Then $\left(S_n^b\right)^{\alpha_n} \limj M_{\nu}$. Convergence holds $\P$-a.s. for $p>5$ and in $\P$-probability for $p=2,3,4$. For $p=5$ it holds $\P$-a.s. if $c \in \left(0,  \frac 14\right)$ and in $\P$-probability else.
\end{theorem}
\begin{remark}Theorem \ref{p:main} immediately implies that $(S_n)^{\alpha_n} \limm M_{\nu}$ on $D([0,\infty))$ equipped with the weaker $M_1$- topology.
\end{remark}

In \cite{BAGun11}
an analogous result is proven in law with respect to the environment for similar conditions on the sequence $\gamma_n$ and fixed $\b$.

Let us comment on the  conditions on $\gamma_n$ and $\beta_n$ in Theorem \ref{p:main}.  They guarantee that 
 $\alpha_n\downarrow 0$ as $n\rightarrow \infty$, and that both sequences $a_n$ and $c_n$ diverge as $n\rightarrow \infty$.
Note here that different choices of the sequence $\beta_n$ correspond to different time scales $c_n$. If $\beta_n\rightarrow  \beta>0$,  as $n\rightarrow\infty$, then $c_n$ is sub-exponential in $n$, while in the case of diverging $\beta_n$, $c_n$ can be as large as 
exponential in $O(n)$.
Finally these conditions guarantee
that the rescaled tail distribution of the $\tau_n$'s, on time scale $c_n$, is regularly varying with index $-\alpha_n$.


We use Theorem \ref{p:main} to derive the limiting behavior of the time correlation function $\CC_n^{\varepsilon}(t,s)$ which, for $t>0$, $s>0$, and $\varepsilon\in(0,1)$ is given by
\be\label{1.26a}
\CC_n^{\varepsilon}(t,s) \equiv  \PP_{\pi_n}\left(A_n^{\varepsilon}(t,s)\right) \ ,
\ee
where $A_n^{\varepsilon}(t,s)\equiv\left\{R_n\left(X_n(t^{1/\alpha_n} c_n), X_n((t+s)^{1/\alpha_n} c_n)\right)\geq 1-\varepsilon\right\}$.

\begin{theorem}
  \label{p:corr}
Under the assumptions of Theorem \ref{p:main}, 
\be\label{1.28}
\lim_{n\rightarrow\infty} \CC_n^{\varepsilon}(t,s) =\frac{t}{t+s}, \quad \forall \varepsilon\in(0,1),\  t,s>0 .
\ee
Convergence holds $\P$-a.s. for $p> 5$ and in $\P$-probability for $p=2,3,4$. For $p=5$ it holds $\P$-a.s. if $c \in \left(0,  \frac 14\right)$ and in $\P$-probability else.
\end{theorem}
Theorem \ref{p:corr} establishes extremal ageing as defined in \cite{BAGun11}. Here, de-correlation takes place on time intervals of the form $[t^{1/\alpha_n}, (t+s)^{1/\alpha_n}]$, while in normal ageing it takes place on time intervals of the form $[t,t+s]$.

The remainder of the paper is organized as follows. We prove the results of Section \ref{S11} in Section \ref{S2}. 
Section \ref{S3} is devoted to the proofs of the statements of Section \ref{S12}. Finally, an additional lemma is proven in the Appendix.

\section{Proofs of the main Theorems} \label{S2}
Now we come to the proofs of the theorems of Section \ref{S11}. 
The proof of Theorem \ref{general} 
hinges on the property
that extremal processes can be constructed from Poisson point processes. Namely, if $\xi'=\sum_{k\in\N}\delta_{\{t'_k,x'_k\}}$ is a Poisson point process on $(0,\infty)\times(0,\infty)$ with mean measure $dt \times d\nu'$, where $\nu'$ is a $\s$-finite measure such that $\nu'(0,\infty)=\infty$, then
\be\label{2.1}
M(t)\equiv \sup\{x'_k: \ t'_k \leq t\}, \quad t>0,
\ee
is an extremal process with $1$-dimensional marginal
\be
F_t(u)=e^{-t\nu'(u,\infty)}.
\ee
(See e.g. \cite{Re}, Chapter 4.3.).
This was used in \cite{DR78} to derive convergence of maxima of random variables to extremal processes
from an underlying Poisson point process convergence.  
Our proof exploits similar ideas and the key fact that the $1/\alpha_n$-norm converges to the sup norm as $\alpha_n\downarrow 0$.

\begin{proof}[Proof of Theorem \ref{general}]  Consider the sequence of point processes defined on $(0,\infty)\times(0,\infty)$ through
\be\label{1.3a}
\xi_n\equiv \sum_{k\in \N} \delta_{\left\{k/a_n, Z_{n,k}^{\alpha_n}\right\}} .
\ee
By Theorem 3.1 of \cite{DR78}, Conditions  \eqref{1.4} and \eqref{1.5} immediately imply  that
$\xi_n\stackrel{n\rightarrow\infty}{\Rightarrow} \xi$, where  $\xi$ is a Poisson point process with intensity measure $dt\times d\nu$.

The remainder of the proof can be summarized as follows. In the first step we construct $(S_n(t))^{\alpha_n}$ from $\xi_n$ by taking the $\alpha_n^{th}$ power of the sum over all points $Z_{n,k}$ up 
to time $\lfloor a_n t\rfloor$. To this end we introduce a truncation threshold $\delta$ and split the ordinates of $\xi_n$ into 
\be\label{truncation}
Z_{n,k}^{\alpha_n} = Z_{n,k}^{\alpha_n}\1_{Z_{n,k}^{\alpha_n}\leq \delta}+ Z_{n,k}^{\alpha_n}\1_{Z_{n,k}^{\alpha_n}> \delta}.
\ee
Applying a summation mapping to $ Z_{n,k}^{\alpha_n}\1_{Z_{n,k}^{\alpha_n}> \delta}$, we show that the resulting process converges to the supremum 
mapping of a truncated version of $\xi$. 
More precisely, let $\delta>0$. Denote by ${\cal M}_p$ the space of point measures on $(0,\infty)\times(0,\infty)$. For $n \in \N$ let $T_n^{\delta}$ be the functional on ${\cal M}_p$, whose value at $m= \sum_{k \in \N} \delta_{\{ t_k, j_k\}}$ is 
\be\label{tn}
 (T_n^{\delta} m)(t)=\textstyle\left(\sum_{t_k \leq t} j_k^{1/\alpha_n} \1_{\left\{j_k>\delta\right\}}\right)^{\alpha_n}, \quad t>0.
\ee
Let $T^{\delta}$ be the functional on ${\cal M}_p$ given by
\be\label{t}
(T^{\delta} m)(t)= \sup\left\{j_k \1_{\left\{j_k>\delta\right\}}: t_k\leq t\right\}, \quad t>0\ .
\ee
We show that $T_n^{\delta} \xi_n \limj T^{\delta} \xi$ as  $n\rightarrow \infty$.

In the second step we prove that the small terms, as $\delta\rightarrow 0$ and $n\rightarrow\infty$, do not contribute to $(S_n)^{\alpha_n}$, i.e.
that for $\varepsilon>0$
\be \label{tightness}
\lim_{\delta \rightarrow 0} \limsup_{n \rightarrow \infty} \PP\left(\rho_{\infty}\left(T_n^{\delta} \xi_n ,S_n^{\alpha_n}\right)>\varepsilon\right)=0 \ ,
\ee
where $\rho_{\infty}$ denotes the Skorokhod metric on $D([0,\infty))$. Moreover, observe that $T^{\delta} \xi \limj M$ as $\delta\rightarrow 0$. Then, by Theorem 4.2 from \cite{Bill68}, the assertion of Theorem \ref{general} follows.

\noindent\emph{Step $1$:} To prove that $T_n^{\delta} \xi_n \limj T^{\delta} \xi$ as  $n\rightarrow \infty$ we use a continuous mapping theorem, namely Theorem 5.5 from \cite{Bill68}. Since the mappings $T_n^{\delta}$ and $T^{\delta}$ are measurable, it is sufficient to show that the  set
\be \label{nulset}
{\cal E}=\left\{m \in {\cal M}_p: \exists\ \left(m_n\right)_{n\in\N} \mbox{ s.t. } m_n \stackrel{v}{\rightarrow} m,
 \ \mbox{ but } \ T_n^{\delta} m_n \ \cancel{ \limj }\ T^{\delta} m\right\},
\ee
where $\stackrel{v}{\rightarrow}$ denotes vague convergence in ${\cal M}_p$,  is a null set with respect to the distribution of $\xi$.
For the Poisson point process $\xi$ it is enough to show that $\PP_{\xi}\left({\cal E}^c \cap {\cal D}\right)=1$, where
\be
{\cal D}\equiv\left\{m\in {\cal M}_p : m\left(\left(0,t\right]\times \left[j,\infty\right)\right)<\infty \ \forall t, j >0\right\}.
\ee
Let  ${\cal C}_{T^{\delta}}\equiv\left\{t >0 : \ \PP_{\xi}\left(\left\{m : \ T^{\delta} m \left(t\right)=T^{\delta} m\left(t-\right)\right\}\right)=1\right\}$ be the set of continuity points of $\xi$.
By definition of the Skorokhod metric, we consider $m \in {\cal D}$, $a, b \in {\cal C}_{T^{\delta}}$, and $\left(m_n\right)_{n \in\N}$ such that $m_n \stackrel{v}{\rightarrow} m$ and show that
\be
\lim_{n\rightarrow\infty}\rho_{\left[a,b\right]}\left(T_n^{\delta} m_n, T^{\delta} m\right) =0 \ ,
\ee
where $\rho_{\left[a,b\right]}$ denotes the Skorokhod metric on $\left[a,b\right]$. Since $m \in {\cal D}$, there exist continuity points $x,y$ of $m$ such that $m((a,b)\times(\delta,\infty))= m((a,b)\times(x,y))<\infty$. Then, Lemma 2.1 from \cite{MO76} yields that $m_n$ also has this property for large enough $n$. Moreover, the points of $m_n$ in $(a,b)\times(x,y)$ converge to the ones of $m$ (cf. Lemma I.14 in \cite{Nev76}). Finally, we use that $\alpha_n \downarrow 0$ as $n \rightarrow \infty$ and thus $T_n^{\delta}$ can be viewed as the $1/\alpha_n$-norm, which converges as $n\rightarrow \infty$ to the sup-norm $T^{\delta}$. Therefore, $T_n^{\delta} \xi_n \limj T^{\delta} \xi$ as $n\rightarrow \infty$.


\noindent\emph{Step $2$:} We prove \eqref{tightness} by showing that the assertion holds true for the Skorokhod metric on $D([0,k])$ for every $k \in \N$. Assume without loss of generality that $k=1$. Let $\varepsilon>0$. We have that
\bea\label{1.13}
&\PP&\textstyle{\left(\sup_{0 \leq t \leq 1}\left|T_n^{\delta} \xi_n \left(t\right) - S_n^{\alpha_n}\left(t\right)\right| > \varepsilon\right) \nonumber}\\
&=& \textstyle{\PP\left(\sup_{0 \leq t \leq 1}\left|\left(\sum_{i=1}^{\lfloor a_n t\rfloor} Z_{n,i} \1_{Z_{n,i} > \delta^{1/\alpha_n}}\right)^{\alpha_n} - 
\left(\sum_{i=1}^{\lfloor a_n t\rfloor} Z_{n,i} \right)^{\alpha_n}\right| > \varepsilon\right)} \ .
\eea 
Since for $n$ large enough $\alpha_n<1$, we know by Jensen inequality that 
\be\label{1.14}
\textstyle{\left|\left(\sum_{i=1}^{\lfloor a_n t\rfloor} Z_{n,i} \1_{Z_{n,i} > \delta^{1/\alpha_n}}\right)^{\alpha_n} - 
\left(\sum_{i=1}^{\lfloor a_n t\rfloor} Z_{n,i} \right)^{\alpha_n}\right| \leq \left|\sum_{i=1}^{\lfloor a_n t\rfloor} Z_{n,i}\1_{Z_{n,i} \leq \delta^{1/\alpha_n}}\right|^{\alpha_n} }\ ,
\ee 
and therefore
\be\label{1.14a}
\eqref{1.13}\leq\textstyle{\PP\left(\sup_{0 \leq t \leq 1}\left|\sum_{i=1}^{\lfloor a_n t\rfloor} Z_{n,i}\1_{Z_{n,i} \leq \delta^{1/\alpha_n}}\right|^{\alpha_n} > \varepsilon\right)\ .
}
\ee
All summands are non-negative. Hence the supremum is attained for $t=1$. Applying a first order Chebychev and Jensen inequality, we obtain that \eqref{1.14a} is bounded above by
\be\label{1.15}
\textstyle{ \e^{-1} \left(\sum_{i=1}^{a_n} \EE \1_{Z_{n,i}\leq \delta^{1/ \alpha_n}} Z_{n,i}\right)^{\alpha_n}
= \frac{\delta}{\e} \left(\sum_{i=1}^{a_n} \EE \1_{Z_{n,i}  \leq \delta^{1/\alpha_n}} \delta^{-1/ \alpha_n} Z_{n,i} \right)^{\alpha_n} \ .}
\ee
By \eqref{1.6} the sum is bounded in $n$ and hence, as $\delta \rightarrow 0$, \eqref{1.15} tends to zero. This concludes the proof of Theorem \ref{general}.
\end{proof}


\begin{proof}[Proof of Theorem \ref{thmblock}]
Throughout we fix a realisation $\o\in\O$ of the random environment but do not make this explicit in the notation. 
We set
\be
\wh S^b_n(t)\equiv S_n^b(t)-
c_n^{-1}\lambda_n^{-1}(J_n(0))e_{n,0} , \ \quad t>0.
\Eq(1.14b)
\ee
$(S_n^b(t))^{\alpha_n} $ differs from $(\wh{S}_n^b(t))^{\alpha_n}$ by one term. All terms in $(S_n^b(t))^{\alpha_n} $ are non-negative and therefore we conclude by Jensen inequality that, for $n$ large enough,
\be\label{1.14ba}
\wh{S}_n^b(t)^{\alpha_n} \leq S_n^b(t)^{\alpha_n} \leq \wh{S}_n^b(t)^{\alpha_n} + \left(c_n^{-1}\lambda^{-1}_n(J_n(0))e_{n,0}\right)^{\alpha_n} \ .
\ee
By Condition (0) the contribution of the term $\left(c_n^{-1}\lambda^{-1}_n(J_n(0))e_{n,0}\right)^{\alpha_n}$ is negligible. Thus we must show that  under Conditions (1)-(3), $(\wh S^b_n)^{\alpha_n}\limj  M_\nu$.
Recall that $k_n(t)\equiv\lf \lf a_n t\rf/\theta_n\rf $ and that for $i\geq 1$,
\be
\textstyle{Z_{n,i} \equiv \sum_{j=\theta_n(i-1)+1}^{\theta_ni}c_n^{-1}\l_n^{-1}(J_n(j))e_{n,j} .}
\Eq(1.14c)
\ee
We apply Theorem \thv(general)  to the $Z_{n,i}$'s.
It is shown in the proof of Theorem 1.2 in \cite{BG10} that Conditions (1) and (2) imply \eqref{1.4} and \eqref{1.5}. It remains to prove that Condition (3) yields \eqref{1.6}. Note that for all $i\geq 1$ and all $(i-1)\theta_n+1 \leq j \leq i \theta_n$,
\be\label{1.14d}
\textstyle{\1_{\{\sum_{j=(i-1)\theta_n+1}^{i \theta_n} \lambda^{-1}_n(J_n(j)) e_{n,j} \leq c_n \delta^{1/\alpha_n}\}} \leq 
\1_{\{\lambda^{-1}_n(J_n(j)) e_{n,j} \leq c_n \delta^{1/\alpha_n}\}}}\ .
\ee
Using \eqref{1.14d}, we observe that \eqref{1.6} is in particular satisfied if for all $\delta>0$ and $t>0$
\be\Eq(1.14ca)
\limsup_{n\rightarrow \infty}\Bigl(\sum_{i=1}^{\lf a_n t\rf}  \EE_{\mu_n} \1_{\{\lambda^{-1}_n(J_n(j)) e_{n,j} \leq c_n \delta^{1/\alpha_n}\}}\delta^{-1/\alpha_n}c_n^{-1}\l_n^{-1}(J_n(j))e_{n,j} \Bigr)^{\alpha_n}<\infty\ ,
\ee
which is nothing but Condition (3). This concludes the proof of Theorem \ref{thmblock}.
\end{proof}

Finally, having Theorem \ref{thmblock} and the results from \cite{BG10}, Theorem \ref{thmblock1} is deduced readily.

\begin{proof}[Proof of Theorem \ref{thmblock1}]
Let $\mu_n$ be the invariant measure $\pi_n$ of the jump chain $J_n$. By Proposition 2.1 of \cite{BG10} we know that Conditions (0), (1-1), and (2-1)
 imply Conditions (0)-(2) of Theorem \ref{thmblock}. Moreover, since $\mu_n=\pi_n$, Condition (3-1) is Condition (3). Thus, the conditions of 
Theorem \ref{thmblock} are satisfied under the assumptions of Theorem \ref{thmblock1} and this yields the claim.
\end{proof}

\section{Application to the $p$ spin SK model} \label{S3}
This section is devoted to the proof of Theorem \ref{p:main}. We show that the conditions of Theorem \ref{thmblock1} are satisfied for the particular
 choices of the sequences $a_n$, $c_n$, $\theta_n$, and $\alpha_n$.

The following lemma from \cite{G10b} (Proposition 3.1) implies that Condition (1-1) holds true for $\theta_n = 3 n^2$.
\begin{lemma}\label{lemma2.1}
Let $P_{\pi_n}$ be the law of the simple random walk on $\S_n$ started in the uniform distribution. Let $\theta_n=3 n^2$. Then, for any $x,y \in \S_n$, and any $i\geq 0$,
\be\label{3.1}
\left| \sum_{k=0}^1P_{\pi_n}\left(J_n(\th_n+i+k)=y,J_n(0)=x\right)
-2\pi_n(x)\pi_n(y)\right|\leq 2^{-3n+1}.
\ee
\end{lemma}

The proof of Condition (2-1) comes in three parts. We first show that $\E\nu_n^t(u,\infty)$ converges to $t\nu(u,\infty)$. 
Next we prove that $\P$-almost surely, respectively in $\P$-probability, the limit of $\nu_n^t(u,\infty)$ concentrates  for all $u>0$ and 
all $t>0$ around its expectation. Lastly we verify that the second part of Condition (2-1) is satisfied in the same convergence mode with respect to the random environment.

\subsection{Convergence of $\E\nu_n^t(u,\infty)$.} \label{S31}
\begin{proposition}\label{prop2.1a}
For all $u>0$ and $t>0$
\be\label{3.2}
\lim_{n\rightarrow \infty} \E \nu_n^t(u,\infty)= \nu^t (u, \infty)\equiv K_p t u^{-1} \ .
\ee
\end{proposition}
The proof of Proposition \ref{prop2.1a} centers on the following key proposition.
\begin{proposition}\label{prop2.2}
Let for $t>0$ and an arbitrary sequence $u_n$,
\be\label{3.3}
\bar{\nu}_n^t(u_n,\infty) = k_n(t)\ \PP_{\pi_n}\Bigl(\max_{i=1,\ldots,\theta_n} \lambda^{-1}_n(J_n(i)) e_{n,i}> u_n^{1/\alpha_n}c_n\Bigr) \ .
\ee
Then, for all $u>0$ and $t>0$,
\be\label{3.5}
\lim_{n\rightarrow\infty}\E\ \bar{\nu}_n^t(u,\infty) = \nu^t(u,\infty) \ .
\ee
The same holds true when $u$ is replaced by $u_n=u\ \theta_n^{-\alpha_n}$.
\end{proposition}
\begin{proof}[Proof of Proposition \ref{prop2.1a}]
By definition, $\nu_n^t(u,\infty)$ is given by
\be\label{3.5a}
\nu_n^t(u,\infty)=k_n(t)\ \PP_{\pi_n}\Bigl(\sum_{i=1}^{\theta_n} \lambda^{-1}_n(J_n(i)) e_{n,i}> u^{1/\alpha_n}c_n\Bigr) .
\ee
The assertion of Proposition \ref{prop2.1a} is then deduced from Proposition \ref{prop2.2} using the upper and lower bounds
\be
\bar{\nu}_n^t(u,\infty) \leq \nu_n^t(u,\infty) \leq \bar{\nu}_n^t(u \theta_n^{-\alpha_n},\infty)  \ .
\label{3.4}
\ee
\end{proof}
The proof of Proposition \ref{prop2.2}, which is postponed to the end of this section, relies on three Lemmata. 
In Lemma \ref{lemma2.3} we show that \eqref{3.5} holds true if we replace the underlying Gaussian process by a simpler Gaussian process $H^1$. 
Lemma \ref{lemma2.4} yields 
\eqref{3.5}  for the maximum over a properly chosen random subset of indices of $H^1$. We use Lemma \ref{lemma2.6} to conclude the proof of Proposition \ref{prop2.2}. 

We start by introducing the Gaussian process $H^1$. Let $v_n$ be a sequence of integers, where each member is of order $n^{\omega}$ for $\omega \in \left(c+\frac 1 2, 1\right)$. Then, $H^1$ is a centered Gaussian process defined on the probability space $(\O,\FF, \P)$ with covariance structure
\bea\label{3.6}
\Delta^1_{i,j} = \begin{cases} 1-2pn^{-1}|i-j|,&\,\hbox{\rm if }\, \lf i/v_n\rf =\lf
j/v_n\rf,\\
0,&\,\hbox{\rm else}.
\end{cases}
\eea
For a given process $U=\{U_i,\ i\in \N\}$ on $(\O,\FF, \P)$ and an index set $I$ define
\be\label{3.7}
\textstyle{F_n(u_n, U,I)\equiv  \P\left(\max_{i\in I} e^{\sqrt{n}\b_n U_i}> u_n^{1/\alpha_n}c_n\right)},
\ee
and for a process $\widetilde U=\{\tilde U_i, \ i\in \N\}$ on $(\O,\FF,\P)$ that may also be dependent on $\FF^J$
\be\label{3.8}
\textstyle{G_n(u_n, \widetilde U,I)\equiv  \PP_{\pi_n} \left(\left.\max_{i\in I} e^{\sqrt{n}\b_n \widetilde U_i}e_{n,i}> u_n^{1/\alpha_n}c_n\right| \cal{F}^{J}\right)}\ .
\ee
\begin{lemma}\label{lemma2.3}
For all $u>0$ and $t>0$
\be\label{3.9}
\lim_{n\rightarrow \infty} k_n(t)\E G_n(u, H^1, [\theta_n])= \nu^t(u,\infty),
\ee
where $[k]\equiv\{1,\ldots, k\}$ for $k \in \N$. The same holds true when $u$ is replaced by $u_n =u\ \theta_n^{-\alpha_n}$.
\end{lemma}
We prove Proposition \ref{prop2.2} and Lemmata \ref{lemma2.3}, \ref{lemma2.4}, and \ref{lemma2.6} for fixed $u>0$ only. 
To show that the claims also hold for $u_n=u \theta_n^{-\alpha_n}$, it is a simple rerun of their proofs, using $\theta_n^{-\alpha_n} \uparrow 1$ 
as $n\rightarrow\infty$.
\begin{proof}
It is shown in Proposition 2.1 of \cite{BAGun11} that, by setting the exponentially distributed random variables to $1$ in \eqref{3.8} and taking expectation with respect to the random environment, we get for all $u>0$ that
\be\label{3.10}
\lim_{n\rightarrow \infty} a_n v_n^{-1}  F_n(u, H^1,[v_n]) = \nu(u,\infty) \ .
\ee
Assume for simplicity that $\theta_n$ is a multiple of $v_n$. Note that blocks of $H^1$ of length $v_n$ are independent and identically distributed. 
Thus,
\bea\label{3.11}
k_n(t)F_n(u, H^1, [\theta_n])&=&k_n(t)\left(1-\left(1-F_n(u, H^1, [v_n])\right)^{\theta_n/v_n}\right)\nonumber\\
&\sim&k_n(t) \theta_n v_n^{-1} F_n(u, H^1, [v_n])\nonumber\\
&\stackrel{n\rightarrow \infty}{\longrightarrow}&\nu^t(u,\infty)\ .
\eea
To show that $k_n(t)\E G_n(u,H^1,[\theta_n])$ also converges to $\nu^t(u,\infty)$  as $n\rightarrow \infty$ we use same arguments as in \eqref{3.11} 
and prove that $a_n v_n^{-1}\E G_n(u,H^1,[v_n]) \rightarrow \nu(u,\infty)$ as $n\rightarrow \infty$. Using Fubini we have that
\bea\label{3.12}
\frac{a_n }{v_n}\E G_n(u, H^1,[v_n])&=& \frac{a_n }{v_n}\int_{c_n u^{1/\alpha_n}}^{\infty}dz \int_0^{\infty}dy \frac{ f_{\max_{i\in[v_n]} e_{n,i}}(y)}{y} f_{\max_{i\in[v_n]} e^{\b_n \sqrt{n} H^1(i)}}(\tfrac zy) \nonumber\\
&=& \frac{a_n }{v_n}\int_0^{\infty}dy  f_{\max_{i\in[v_n]} e_{n,i}}(y) F_n(u\ y^{-\alpha_n},H^1,[v_n])\ ,
\eea
where $f_{Z}(\cdot)$ denotes the density function of $Z$. Since we want to use computations from the proof of Proposition 2.1 in \cite{BAGun11}, 
it is essential that the integration area over $y$ is bounded from below and above. 
We bound \eqref{3.12} from above by
\bea
\eqref{3.12} &\leq&
a_n v_n^{-1} \PP\Bigl(\max_{i=1,\ldots,v_n} e_{n,i} \leq e^{-n v_n^{-1-\delta}}\Bigr) \label{3.13}\\
&+& a_n v_n^{-1} \int_{e^{-n v_n^{-1-\delta}}}^{e^{n v_n^{-1/2-\delta}}}dy f_{\max_{i\in[v_n]} e_{n,i}}(y) F_n(u\ y^{-\alpha_n},H^1,[v_n])\label{3.14} \\
&+& 
a_n v_n^{-1} \PP\Bigl(\max_{i=1,\ldots,v_n} e_{n,i} > e^{n v_n^{-1/2-\delta}}\Bigr) \label{3.15}\ ,
\eea
where $\delta>0$ is chosen in such a way that $n v_n^{-1-\delta}$ diverges and $v_n^{\delta} \gamma_n^2 \downarrow 0$ as $n\rightarrow \infty$, i.e. $\delta<\min\left\{2c, \frac{1-\omega}{\omega}\right\}$.
Then,
\be \label{3.16}
\eqref{3.13}=
a_n v_n^{-1}\Bigl(1-\exp\Bigl(-e^{-n v_n^{-1-\delta}}\Bigr)\Bigr)^{v_n}\leq a_n e^{-n v_n^{-\delta}} = o\left(e^{-n v_n^{-\delta}(1- \gamma_n^2 v_n^{\delta})}\right) \ ,
\ee
i.e. \eqref{3.13} vanishes as $n\rightarrow \infty$.
Similarly,
\be\label{3.17}
 \eqref{3.15} =a_n v_n^{-1}\Bigl(1-\Bigl(1-\exp\Bigl(-e^{-n v_n^{-1/2-\delta}}\Bigr)\Bigr)^{v_n}\Bigr)= o\Bigl(e^{\gamma_n^2 n - e^{n v_n^{-1/2-\delta}}}\Bigr) \stackrel{n\rightarrow \infty}{\longrightarrow} 0 \ .
\ee

As in equation (2.31) in \cite{BAGun11} we see that \eqref{3.14} is given by
\be\label{3.18}
\int_{e^{-n v_n^{-1-\delta}}}^{e^{n v_n^{-1/2-\delta}}} dy \frac{f_{\max_{i\in[v_n]} e_{n,i}}(y)}{\gamma_n^{2}v_n}\sum_{k=1}^{v_n}\int_{D_k^{''}} da_2 \cdots da_{v_n} \int_{\log (u y^{- \alpha_n })}^{\infty}da_1
\frac{e^{-h_k(a_1,\ldots,a_{v_n})}}{(2\pi)^{\frac{v_n-1}{2}}} \ ,
\ee
where for $k\in\{1,\ldots,v_n\}$
\be\label{3.19}
\textstyle{h_k(a_1,\ldots,a_{v_n}) = a_1-\frac{a_1^2 C_1 }{\gamma_n^{2} n}-\frac 1 2 \sum_{i=2}^{v_n} a_i^2+\frac{(a_2+\ldots+a_k - a_{k+1}-\ldots -a_{v_n})a_1 C_2} {\gamma_n n} \ ,}
\ee
for some constants $C_1, C_2 >0$ and a sequence of sets $D_k^{''}\subseteq \R^{v_n-1}$ such that 
\be\label{3.20}
\gamma_n^{-2}v_n^{-1}\sum_{k=1}^{v_n}\int_{D_k^{''}} da_2\cdots da_{v_n} (2\pi)^{-v_n/2-1/2}
e^{-\frac 1 2 \sum_{i=2}^{v_n} a_i^2}  \stackrel{n\rightarrow \infty}{\longrightarrow} K_p \ .
\ee
The aim is to separate $a_1$ from $a_2,\ldots, a_{v_n}$ in \eqref{3.19}. We bound the mixed terms in $e^{-h_k}$ up to an exponentially small 
error by $1$. This can be done using a large deviation argument for $|a_2 +\ldots+a_{v_n}|$ together with the fact that $|\log y| \in \left[n v_n^{-1-\delta}, n v_n^{-1/2-\delta}\right]$.
Computations yield together with the bounds in \eqref{3.18}-\eqref{3.20} that, up to a multiplicative error that tends to $1$ as $n\rightarrow\infty$ exponentially fast, \eqref{3.14} is bounded from above by
\be\label{3.21}
\int_{e^{-n v_n^{-1-\delta}}}^{\infty} dy f_{\max_{i\in[v_n]} e_{n,i}}(y)  y^{\alpha_n}\ u^{-1} K_p \leq\nu(u,\infty) \int_{0}^{\infty}dy f_{\max_{i\in[v_n]} e_{n,i}}(y)  y^{\alpha_n}\ .
\ee
Moreover by Jensen inequality,
\bea \label{3.22}
\eqref{3.21}&\leq&\nu(u,\infty) \Bigl(\EE_{\pi_n} \max_{i\in [v_n]} e_{n,i}\Bigr)^{\alpha_n} \nonumber\\
&=&\nu(u,\infty) \Bigl(\int_0^{\infty}dy\ \PP\Bigl(\max_{i\in [v_n]} e_{n,i}> y\Bigr)\Bigr)^{\alpha_n} \nonumber\\
&=&\nu(u,\infty) \Bigl(\int_0^{\infty} dy \Bigl(1-\left(1-e^{-y}\right)^{v_n}\Bigr)\Bigr)^{\alpha_n} \nonumber\\
&\leq& \nu(u,\infty) v_n^{\alpha_n} \ , 
\eea
which, as $n\rightarrow \infty$, converges to $\nu(u,\infty)$.

To conclude the proof of \eqref{3.9}, we bound \eqref{3.12} from below by
\be\label{3.23}
\eqref{3.12} \geq \frac{a_n}{v_n}\int_0^{\infty} dy f_{ e_{n,1}}(y) F_n(u\ y^{-\alpha_n},H^1,[v_n]) \ .
\ee
To show that the right hand side of \eqref{3.23} is greater than or equal to $\nu(u,\infty)$, one proceeds as before.
\end{proof}

In the following we form a random subset of $[\theta_n]$ in such a way that on the one hand, with high probability, it contains the maximum of $e^{\b_n \sqrt{n} H^1(i)}$ over all $i\in[\theta_n]$ . On the other hand it should be a sparse enough subset of $[\theta_n]$ so that we are able to de-correlate the random landscape and deal with the SK model. This dilution idea is taken from \cite{BAGun11}. 

If the maximum of $e^{\b_n \sqrt{n} H^1(i)}$ crosses the level $c_n u^{1/\alpha_n}$, then it will typically be much larger than $c_n u^{1/\alpha_n}$ so that, due to strong correlation, at least $\gamma_n^{-2}$ of its direct neighbors will be above the same level. To see this, we consider Laplace transforms. Set for $v>0$
\be\label{3.24}
\textstyle{
\wh{F}_n(v, H^1, \theta_n)\equiv\int_0^{\infty} dz\ e^{-zv} \P\left(\delta_n \sum_{i=1}^{\theta_n} \1_{e^{\b \sqrt{n} H^1(i)} >c_n u^{1/\alpha_n}} >z\right) \ ,}
\ee
where $\delta_n\in [0,1]$ for every $n \in \N$.
We have that 
\bea
\wh{F}_n(v, H^1, \theta_n)&=&\textstyle{ \frac 1 v\left(1-\E \exp\left(-\delta_n \sum_{i=1}^{\theta_n} \1_{e^{\b_n \sqrt{n} H^1(i)} >c_n u^{1/\alpha_n}}\right)\right)}\nonumber\\
&=&\textstyle{\frac 1 v \left(1-\left(\E \exp\left(-\delta_n \sum_{i=1}^{v_n} \1_{e^{\b_n \sqrt{n} H^1(i)} >c_n u^{1/\alpha_n}}\right)\right)^{\theta_n/v_n}\right)} . \label{3.25}
\eea
From \cite{BAGun11}, Proposition 1.3, we deduce that for the choice $\delta_n=\gamma_n^2\rho_n$, where $\rho_n$ is any diverging sequence of order $O(\log n)$,
\be\label{3.26}\textstyle{
\lim_{n\rightarrow\infty} a_n v_n^{-1} \left(1-\E \exp\left(-\delta_n \sum_{i=1}^{v_n} \1_{e^{\b_n \sqrt{n} H^1(i)} >c_n u^{1/\alpha_n}}\right)\right) = \nu(u,\infty) \ .}
\ee
Therefore we have for the same choice of $\delta_n$ that
\be\label{3.27}
k_n(t)\wh{F}_n(v, H^1, \theta_n) \rightarrow t v^{-1} \nu(u,\infty) \ .
\ee
From this we conclude that if the maximum is above the level $c_n u^{1/\alpha_n}$ then immediately $O(\gamma_n^{-2})$ are above this level. More precisely, we obtain
\begin{lemma}\label{lemma2.4}
Let $\rho_n$ be as described above. Let $\{\xi_{n,i}: \ i \in \N, \ n\in \N\}$ be an array of row-wise independent and identically distributed Bernoulli random variables such that $\P(\xi_{n,i}=1)= 1-\P(\xi_{n,i}=0)=\gamma_n^2 \rho_n$, and such that $\{\xi_{n,i}: i\in\N,\ n\in\N\}$ is independent of everything else. Set
\be\label{3.28}
\cal{I}_{k}= \{i  \in \{1,\ldots, k\}: \ \xi_{n,i}=1\} \ .
\ee
Then, for all $u>0$ and $t>0$
\be\label{3.29}
\lim_{n\rightarrow \infty} k_n(t) \E G_n(u, H^1, \cal{I}_{\theta_n}) = \nu^t(u,\infty) \ .
\ee
The same holds true when $u$ is replaced by $u_n=u\ \theta_n^{-\alpha_n}$.
\end{lemma}
\begin{proof}
It is shown in Lemma 2.3 of \cite{BAGun11} that
\be\label{3.30}
\lim_{n\rightarrow \infty} a_n v_n^{-1} F_n(u, H^1, \cal{I}_{v_n}) = \nu(u,\infty) \ .
\ee
Since the random variables $\xi_{n,i}$ are independent, the claim of Lemma \ref{lemma2.4} is deduced by the same arguments as in \eqref{3.11}.\end{proof}
To conclude the proof of Proposition \ref{prop2.2}, we use a Gaussian comparison result. The following lemma is an adaptation of Theorem 4.2.1of \cite{LLR83}.
\begin{lemma}\label{lemma2.5}
Let $H^0$ and $H^1$ be Gaussian processes with mean $0$ and covariance matrix $\Delta^0=(\Delta^0_{ij})$ and $\Delta^1=(\Delta^1_{ij})$, respectively. Set $\Delta^m\equiv\left(\Delta^m_{ij}\right)=\left(\max\{\Delta^0_{ij}, \Delta^1_{ij}\}\right)$ and $\Delta^h \equiv h \Delta^0 + (1-h) \Delta^1$, for $h\in[0,1]$. Then, for $s\in\R$,
\bea\label{3.31}
&&\textstyle{\P(\max_{i\in I} H^0(i)\leq s)-\P(\max_{i\in I} H^1(i)\leq s) }\nonumber\\
&\leq&\textstyle{ \sum_{i,j\in I} (\Delta_{ij}^0-\Delta_{ij}^1)^+ \exp\left(-\frac{s^2}{1+\Delta^m_{ij}}\right)  \int_0^1 dh (1-(\Delta_{ij}^h)^2)^{-\frac1 2}}\ ,
\eea
where $(x)^+\equiv\max\{0,x\}$.
\end{lemma}
We use Lemma \ref{lemma2.5} to prove that
\begin{lemma}\label{lemma2.6}
Let $H^0$ be given by $H^0(i)\equiv n^{-1/2} H_n(J_n(i))$, $i\in \N$. For all $u>0$ and $t>0$
\be\label{3.32}
\lim_{n\rightarrow \infty} k_n(t) E_{\pi_n} |\E G_n(u, H^0, \theta_n) -  \E G_n(u,H^1,\theta_n)| =0.
\ee
The same holds true when $u$ is replaced by $u_n=u \theta_n^{-\alpha}$.
\end{lemma}
\begin{proof}
The proof is in the same spirit as that of Proposition 3.1 in \cite{BAGun11}. Together with Lemma \ref{lemma2.4}, it is sufficient to show that 
\be\label{3.33}
k_n(t) E_{\pi_n}(\E G_n(u, H^1, [\theta_n]) -  \E G_n(u, H^0, [\theta_n]))^{+} \rightarrow 0
\ee
and
\be\label{3.34}
k_n(t) E_{\pi_n}|\E G_n(u, H^1, \cal{I}_{\theta_n}) -  \E G_n(u,H^0,\cal{I}_{\theta_n})| \rightarrow 0\ .
\ee
We do this by an application of Lemma \ref{lemma2.5}. Let $\hat s_n$ be given by
\be\label{3.35}
\textstyle{
\hat s_n=\frac{1}{\sqrt{n}\beta_n}\left(\log c_n + \frac{\beta_n}{\gamma_n} \log u- \max_{i\in[\theta_n]} \log e_{n,i}\right)\ .}
\ee
Then we obtain by Lemma \ref{lemma2.5} that
\bea\label{3.36}
&&\eqref{3.33} \nonumber\\
&= &\textstyle{k_n(t) E_{\pi_n}\left(\E \EE_{\pi_n}  \left[ \1_{\max_{i\in [\theta_n]} H^1(i)\leq \hat s_n} -\1_{\max_{i\in [\theta_n]} H^0(i)\leq \hat s_n}\left.\right|{\cal F}^{J}\right]\right)^+ }\nonumber\\
&\leq& \textstyle{k_n(t) E_{\pi_n}\sum_{i,j\in [\theta_n]} (\Delta_{ij}^1-\Delta_{ij}^0)^+ \EE_{\pi_n} e^{-\hat s_n^2 (1+\Delta^m_{ij})^{-1}} \int_0^1 dh (1-(\Delta_{ij}^h)^2)^{-\frac1 2}}\ .
\eea
To remove the exponentially distributed random variables $e_{n,i}$ in \eqref{3.36}, let $B_n=\{1 \leq \max_{i\in[\theta_n]} e_i \leq n\}$. We have for $s_n=(n^{1/2}\beta_n)^{-1}\left(\log c_n +\tfrac{\beta_n}{\gamma_n} \log u- \log n\right)$ that
\be\label{3.37}
\EE_{\pi_n}\left( \1_{B_n} \exp\left(-\hat s_n^2 (1+\Delta^m_{ij})^{-1}\right) \right)\leq \exp\left(-s_n^2 (1+\Delta^m_{ij})^{-1}\right)\ .
\ee
One can check that $k_n(t)\PP(B_n^c)\downarrow 0$. Moreover, by definition of $s_n$, there exists for all $u>0$ a constant $C<\infty$ such that for $n$ large enough
\be\label{3.38}
\eqref{3.33}\leq
\textstyle{C k_n(t)E_{\pi_n}\sum_{i,j\in [\theta_n]} (\Delta_{ij}^1-\Delta_{ij}^0)^+ e^{-\gamma_n^2 n(1+\Delta^m_{ij})^{-1}}  \int_0^1 dh (1-(\Delta_{ij}^h)^2)^{-\frac 1 2}}\ .
\ee
Likewise we deal with \eqref{3.34}. The terms in \eqref{3.34} are non-zero if and only if $i,j\in \cal{I}_{\theta_n}$. By assumption, the probability of this event is $(\gamma_n^2 \rho_n)^2$. Hence, \eqref{3.34} is bounded above by
\be\label{3.39}
C k_n(t) (\gamma_n^2 \rho_n)^2\textstyle{E_{\pi_n}\sum_{i,j\in [\theta_n]} |\Delta_{ij}^0-\Delta_{ij}^1| e^{-\gamma_n^2 n(1+\Delta^m_{ij})^{-1}}  \int_0^1 dh (1-(\Delta_{ij}^h)^2)^{-\frac 1 2}}\ .
\ee
We divide the summands in \eqref{3.38} and \eqref{3.39} respectively into two parts: pairs of $i,j$ such that $\lf i/v_n\rf \neq\lf j/v_n\rf$ and those such that $\lf i/v_n\rf=\lf j/v_n\rf$. If $\lf i/v_n\rf \neq\lf j/v_n\rf$ then we have by definition of $H^1$ that $\Delta^1_{ij}=0$. For $i,j$ such that $\lf i/v_n\rf=\lf j/v_n\rf$, we have $\Delta^1_{ij}\leq \Delta^0_{ij}$. In view of this, we get after some computations that
\be\label{3.40}
\eqref{3.38}\leq
C k_n(t) E_{\pi_n} \left[\textstyle{\sum_{\lf i/v_n\rf\neq\lf j/v_n\rf}^{\theta_n}(\Delta_{ij}^0)^- e^{-\gamma_n^2 n} }\right],
\ee
and 
\bea\label{3.41}
\eqref{3.39}&\leq&
C k_n(t) \gamma_n^4 \rho_n^2 E_{\pi_n} \left[\textstyle{\sum_{\lf i/v_n\rf\neq\lf j/v_n\rf}^{\theta_n}|\Delta_{ij}^0| e^{-\gamma_n^2 n(1+\Delta^0_{ij})^{-1}} }\right. \nonumber\\ 
&+&\left. \textstyle{\sum_{\lf i/v_n\rf=\lf j/v_n\rf}^{\theta_n}|\Delta_{ij}^0-\Delta_{ij}^1| e^{-\gamma^2 n(1+\Delta^0_{ij})^{-1}}   (1-(\Delta_{ij}^0)^2)^{-\frac 1 2}}
\right] \ .
\eea
Since $(\Delta_{ij}^0)^- =O(n)$ we know by definition of $a_n$ and $\theta_n$ that
\be\label{3.42}
\eqref{3.40}\leq C \theta_n n^{3/2} \alpha_n^{-1} e^{-\frac 1 2 \gamma_n^2 n } \ ,
\ee
which tends to zero as $n \rightarrow \infty$. Thus \eqref{3.33} holds true.

To conclude the proof of \eqref{3.34} we use Lemma \ref{lemma2.7} from the appendix. We get that \eqref{3.39} is bounded above by
\bea\label{3.43a}
\textstyle{\bar C t a_n \sum_{d=0}^{n} e^{-\gamma_n^2 n(1+d)^{-1}}\left(\tfrac{d^2}{v_n n}\1_{d\leq v_n} + \tfrac{\exp(\eta \gamma_n^2 \min\{d,n-d\})}{v_n \gamma_n^2}\right),}
\eea
for some $\bar C<\infty$ and $\eta<\infty$. With the same arguments as in the proof of (3.3) in \cite{BAGun11}, we obtain that \eqref{3.43a} tends to zero as $n \rightarrow \infty$.
\end{proof}
\begin{proof}[Proof of Proposition \ref{prop2.2}]
Observe that
\be\label{3.45}
\left| \E \bar\nu_n^t(u,\infty) - \nu^t(u,\infty)\right|= \left|k_n(t)E_{\pi_n}\E G_n(u,H^0, [\theta_n]) - \nu^t(u,\infty)\right|\ ,
\ee
which is bounded above by 
\be \label{3.46}
k_n(t)E_{\pi_n}\left|\E G_n(u,H^0, [\theta_n]) -\E G_n(u,H^1, [\theta_n])\right| + \left|k_n(t)\E G_n(u,H^1, [\theta_n]) - \nu^t(u,\infty)\right| \ .
\ee
By Lemma \ref{lemma2.3} and Lemma \ref{lemma2.6}, both terms vanish as $n \rightarrow \infty$ and Proposition \ref{prop2.2} follows.
\end{proof}

\subsection{Concentration of $\nu_n^t(u,\infty)$} \label{S32}
To verify the first part of Condition (2-1) we control the fluctuation of $\nu_n^t(u,\infty)$ around its mean.
\begin{proposition}\label{prop2.8}
For all $u>0$ and $t>0$ there exists $C=C(p,t,u)<\infty$, such that
\be\label{3.48}
\E \left(\bar \nu_n^t(u, \infty)- \E\bar \nu_n^t(u, \infty)\right)^2 \leq C \gamma_n^{-2} n^{1-p/2} \ .
\ee
The same holds true when $u$ is replaced by $u_n=u\theta_n^{-\alpha_n}$. In particular, for $p>5$ and $c\in(0,\frac 12)$ or $p=5$ and $c<\frac 14$, the first part of Condition (2-1) holds for all $u>0$ and $t>0$, $\P$-a.s.
\end{proposition}
\begin{proof}
Let $\left\{e'_{n,i} :i \in \N , n \in \N\right\}$ and $J'_n$ be independent copies of $\left\{e_{n,i} :i \in \N , n \in \N\right\}$  and $J_n$ 
respectively. Writing $\pi_n$ for the initial distribution of $J_n$ and $\pi'_n$ for that of $J_n'$, we define
\bea\label{3.56}
\bar{G}_n(u,H^0,[\theta_n]) &\equiv& \textstyle{\PP_{\pi_n}\left(\left.\max_{i\in[\theta_n]} e^{\b_n H_n(J_n(i))} e_{n,i}\leq c_n u^{1/\alpha_n}\right|\cal{F}^{J}\right)}\nonumber\\
\bar{G}_n(u,H^{0'},[\theta_n]) &\equiv&\textstyle{\PP_{\pi'_n}\left(\left. \max_{i\in[\theta_n]} e^{\b_n H_n(J'_n(i))}e'_{n,i}\leq c_n u^{1/\alpha_n}\right|\cal{F}^{J'}\right).}
\eea
Then, as in (3.21) in \cite{BG10},
\bea\label{3.57}
\E\left(\EE_{\pi_n} \bar{G}_n(u,H^0,[\theta_n])\right)^2 &=& \E \EE_{\pi_n}\bar{G}_n(u,H^0,[\theta_n]) \EE_{\pi'_n}\bar{G}_n(u,H^{0'},[\theta_n])\nonumber\\
&=& \EE_{\pi_n}\EE_{\pi'_n}\E \bar G_n(u,V^0,[2\theta_n])\ ,
\eea
where $V^0$ is a Gaussian process defined by
\be\label{3.58}
V^0(i)= \begin{cases} n^{-1/2} H_n(J_n(i)), &\hbox{\rm if}\ 1 \leq i \leq \theta_n,\\
n^{-1/2} H_n(J'_n(i)),&\hbox{\rm if} \ \theta_n+1 \leq i \leq 2 \theta_n \ .
\end{cases}
\ee

To further express $\left(\E \EE_{\pi_n} \bar{G}_n(u,H^0,[\theta_n])\right)^2$, let $V^1$ be a centered Gaussian process with covariance matrix
\be\label{3.59}
\Delta_{ij}^1 = \begin{cases} \Delta_{ij}^0, &\hbox{\rm if}\, \max\{i,j\} \leq \theta_n, \mbox{ or }\min\{i,j\} \geq \theta_n,\\
0,&\hbox{\rm else} ,\ 
\end{cases}
\ee
where $\Delta^0=(\Delta_{ij}^0)$ denotes the covariance matrix of $V^0$. Then, as in (3.23) in \cite{BG10}, 
\be\label{3.60}
\left(\E \EE_{\pi_n} \bar{G}_n(u,H^0,[\theta_n])\right)^2 =\EE_{\pi_n}\EE_{\pi'_n} \E\bar{G}_n(u,V^1,[2\theta_n])\ .
\ee
As in the proof of Lemma \ref{lemma2.6} we use Lemma \ref{lemma2.5} to obtain that
\bea\label{3.61}
&& k_n^2(t)\E \left(\EE_{\pi_n} \bar{G}_n(u,H^0,[\theta_n])- \E \EE_{\pi_n} \bar{G}_n(u,H^0,[\theta_n])\right)^2 \nonumber \\
 &\leq &
2 k_n^2(t) \sum_{{1 \leq i\leq\theta_n }\atop{ \theta_n+1 \leq j\leq2\theta_n}} E_{\pi_n} E_{\pi'_n} \Delta_{ij}^0 e^{-\gamma_n^2 n (1+\Delta^0_{ij})^{-1}}.
\eea
It is shown in (3.29) of \cite{BG10} that
\be\label{3.62}
E_{\pi_n} E_{\pi_n'} \1_{\Delta_{ij}^0 =\left(\frac{m}{n}\right)^p } = 2^{-n} {n\choose (n-m)/2}, \quad \mbox{ for }m\in \{0,\ldots,n\}.
\ee
From this, and with the definition of $a_n$, we have that
\bea\label{3.63}
\eqref{3.61}&\leq& 2 t^2 a_n^2 \sum_{m=0}^n  2^{-n} {n\choose (n-m)/2} \left(\frac m n\right)^p \exp\left( -\frac{\gamma_n^2 n } {1+(\frac m n)^p}\right)\nonumber \\
&\leq& 2 t^2\gamma_n^{-2} \sum_{m=0}^n 2^{-n} n {n\choose (n-m)/2}  \left(\frac m n\right)^p \exp\left(\gamma_n^2 n 
\frac{(\frac m n)^p} {1+(\frac m n)^p}\right)\nonumber\\
&=& 2 t^2\gamma_n^{-2}\sum_{d=0}^n 2^{-n} n {n\choose d}  \left(1-\frac{2d}{n}\right)^p \exp\left(\gamma_n^2 n 
\frac{ (1-\frac{2d}{n})^p} {1+(1-\frac{2d}{n})^p}\right)\nonumber\\
&\leq&2 t^2\gamma_n^{-2} \sum_{d=0}^n n^{1/2}  \left(1-\frac{2d}{n}\right)^p_+ \exp\left(n \Upsilon_{n,p}\left(\tfrac dn\right)\right) 
J_n\left(\tfrac dn\right)\ ,
\eea
where for $u\in(0,1)$ we set $\Upsilon_{n,p}(u)=\gamma_n^2-I(u)-\gamma_n^2 (1+|1-2u|^p)^{-1}$ and 
$J_n(u)=2^{-n} {n\choose \lfloor nu \rfloor} \sqrt{\pi n} e^{n I(u)}$ for $I(u)=u\log u + (1-u)\log(1-u)+\log 2$.
Note that \eqref{3.63} has the same form as (3.28) in \cite{BBC08}. Following the strategy of \cite{BBC08}, we show that there exist $\delta,\delta'>0$
and $c>0$ such that
\be\label{3.63a}
\Upsilon_{n,p}\leq \begin{cases} -c\left(u-\tfrac 12\right)^2 , &\hbox{\rm if}\, u\in(\tfrac 12-\delta, \tfrac 12+\delta),\\
-\delta',&\hbox{\rm else}.
\end{cases}
\ee
Since $\gamma_n=n^{-c}$ this can be done, independently of $p$, as in \cite{BAGun11} (cf. (3.19) and (3.20)). Finally, together with the calculations
 from (3.28) in \cite{BBC08} we obtain that
\be\label{3.64}
\E \left(\bar \nu_n^t(u, \infty)- \E\bar \nu_n^t(u, \infty)\right)^2 \leq C\gamma_n^{-2} n^{1-p/2}.
\ee
The same arguments and calculations are used to prove that \eqref{3.48} also holds when $u$ is replaced by $u_n=u \theta_n^{-\alpha_n}$.
Let $p>5$ and $c\in(0,\frac 12)$ or $p=5$ and $c<\frac 14$. Then, by Borel-Cantelli Lemma, for all $u>0$ and $t>0$ there exists a set $\Omega(u,t)$ with $\P(\Omega(u,t))=1$ such that on $\Omega(u,t )$, for all $\varepsilon>0$ and $n$ large enough, we have that $|\bar\nu_n^t(u,\infty)-\nu^t(u,\infty)|<\varepsilon$ and $|\bar\nu_n^t(u_n,\infty)-\nu^t(u,\infty)|<\varepsilon$. From this we conclude together with \eqref{3.4} that, on $\Omega(u,t)$ and for $n$ large enough,
\be\label{3.65}
\nu^t(u,\infty)-\varepsilon \leq \nu_n^t(u,\infty) \leq \nu^t(u_n,\infty)+\varepsilon ,
\ee
i.e. Condition (2-1) is satisfied, for all $u>0$ and $t>0$, $\P$-a.s.
\end{proof}
\begin{proposition}\label{prop2.9}
Let $p=2,3,4$ and $c\in(0,\frac 12)$ or $p=5$ and $c>\frac 14$. Then, the first part of Condition (2-1) holds in $\P$-probability for all $u>0$ and $t>0$.
\end{proposition}
\begin{proof}
For all $\varepsilon>0$, we bound $\P\left(|\nu_n^t(u,\infty)-\E(\nu_n^t(u,\infty))|>\varepsilon\right)$ from above by
\bea
&&\P\left(|\nu_n^t(u,\infty)-k_n(t)\EE_{\pi_n}G_n(u,H^0,{\cal I}_{\theta_n})|>\varepsilon/3\right)\label{3.66}\\
&+& \P\left(k_n(t)|\EE_{\pi_n}G_n(u,H^0,{\cal I}_{\theta_n})-\E\EE_{\pi_n}G_n(u,H^0,{\cal I}_{\theta_n}) |>\varepsilon/3\right)\label{3.66a}\\
&+& \1_{\{|\E(\nu_n^t(u,\infty)) - k_n(t)\E\EE_{\pi_n}G_n(u,H^0,{\cal I}_{\theta_n})|>\varepsilon/3\}}\label{3.66b}.
\eea
Observe that by a first order Chebychev inequality,
\be\label{3.66c}
\eqref{3.66}\leq |\E\nu_n^t(u,\infty)-k_n(t)\E\EE_{\pi_n}G_n(u,H^0,{\cal I}_{\theta_n})|.
\ee
By Lemmata \ref{lemma2.3}, \ref{lemma2.4}, and \ref{lemma2.6}, \eqref{3.66c} tends to zero as $n \rightarrow \infty$. For the same reason, \eqref{3.66b} is equal to zero for large enough $n$. To bound \eqref{3.66a}, we  calculate the variance of $k_n(t)\EE_{\pi_n}G_n(u,H^0,{\cal I}_{\theta_n})$. As in the proof of Proposition \ref{prop2.8} we use Lemma \ref{lemma2.5}, but take into account that there can only be contributions to the left hand side of \eqref{3.31} if $i,j\in{\cal I}_{\theta_n}$. This gives us the additional factor $\left(\gamma_n^{2}\rho_n\right)^2$ in \eqref{3.61}. Therefore the variance of $k_n(t)\EE_{\pi_n}G_n(u,H^0,{\cal I}_{\theta_n})$ is bounded above by $C (\gamma_n \rho_n)^2 n^{1-p/2}$ which, for all $p\geq 2$, vanishes as $n \rightarrow \infty$. Hence, we have proved Proposition \ref{prop2.9}.
\end{proof}
\subsection{Second part of Condition (2-1)} \label{S33}
We proceed as in Section 3.4 in \cite{BG10} to verify the second part of Condition (2-1) . With the same notation as in \eqref{1.9}, we define for $u>0$ and $t>0$
\bea
\wt \eta_n^t(u)&\equiv& k_n(t) n^{-1}\sum_{x\in \S_n} \left(Q_n^u(x)\right)^2 \ , \label{3.65aa}\\
\eta_n^t(u)&\equiv& k_n(t)\sum_{x\in \S_n} \sum_{x'\in \S_n} \mu_n(x,x') Q_n^u(x) Q_n^u(x') \ , \label{3.66aa} 
\eea
where $\mu_n(\cdot,\cdot)$ is the uniform distribution on pairs $(x,x')\in \S_n^2$ that are at distance $2$ apart, i.e.
\be\label{3.67}
\mu_n(x,x')= \begin{cases} 2^{-n} \frac{2}{n(n-1)} , &\hbox{\rm if}\, \dist(x,x')=2,\\
0,&\hbox{\rm else}.
\end{cases}
\ee
We prove that the expectations of both \eqref{3.65aa} and \eqref{3.66aa} tend to zero. First and second order Chebychev inequalities then yield that the second part of Condition (2-1) holds in $\P$-probability, respectively $\P$-a.s.
\begin{lemma}\label{lemma3.10}
For all $u>0$ and $t>0$
\be\label{3.68}
\lim_{n\rightarrow\infty}\E\wt \eta_n^t(u)=\lim_{n\rightarrow\infty}\E\eta_n^t(u) =0\ .
\ee
\end{lemma}
\begin{proof}
We show that $\lim_{n\rightarrow\infty}\E\eta^t_n(u)=0$. The assertion for $\wt\eta_n^t(u)$ is proved similarly. Let
\be\label{3.69}\textstyle{
\bar Q_n^u(x) \equiv\PP_{x}\Bigl(
\sum_{j=1}^{\theta_n}\lambda^{-1}_n(J_n(j))e_{n,j}\leq c_n u^{1/\alpha_n}\Bigr) \ .}
\ee
Rewrite \eqref{3.66aa} in the following way
\bea
&&\textstyle{k_n(t)\sum_{x\in \S_n} \sum_{x'\in \S_n} \mu_n(x,x')\left(1-\bar Q_n^u(x)\right) \left(1-\bar Q_n^u(x')\right)}\nonumber\\
&=&\textstyle{k_n(t)\left[1-\sum_{(x,x')\in \S_n^2} \mu_n(x,x')\left(\bar Q_n^u(x) + \bar Q_n^u(x')-\bar Q_n^u(x)\bar Q_n^u(x')\right)\right]}\nonumber\\
&=&\textstyle{k_n(t)\left[1- 2 \sum_{x\in \S_n} \pi_n(x)\bar Q_n^u(x) + \sum_{(x,x')\in \S_n^2} \mu_n(x,x')\bar Q_n^u(x)\bar Q_n^u(x')\right]}\label{3.70} \ . \ \ 
\eea
To shorten notation, write
\be\label{3.71}
\textstyle{K_n^u \equiv  \textstyle{ \PP_{\pi_n}\Bigl( \left.\max_{i\in \{\overline{\theta}_n,\ldots,\theta_n\}} e^{\sqrt{n}\b_n H^0(i)}e_{n,i} > c_nu^{1/\alpha_n}\right|\cal{F}^{J}\Bigr)}=\sum_{x \in \S_n} 2^{-n}K_n^u(x) ,}
\ee
where $\overline{\theta}_n\equiv 2n \log n$ and
\be\label{3.72}
\textstyle{ K_n^u(x) \equiv \textstyle{ \PP_{x}\left( \left.\max_{i\in \{\overline{\theta}_n,\ldots,\theta_n\}} e^{\sqrt{n}\b_n H^0(i)}e_{n,i}> c_n u^{1/\alpha_n}\right|\cal{F}^{J}\right)}}.
\ee
Using the bound $\bar Q_n^u(x) \leq \EE_{x}(1-K_n^u(x)) \equiv\EE_{x}\bar K_n^u(x)$, $x \in \S_n$, and taking expectation with respect to the random environment we obtain that
\bea\label{3.73}
\E\eta_n^t(u)&\leq&k_n(t)- 2\left(k_n(t)-\E \nu_n^t (u,\infty)\right) \\
&+& \textstyle{k_n(t) \sum_{(x,x')\in \S^2_n}\mu_n(x,x')\E \left[ \EE_{x}\bar K_n^u(x) \EE_{x'}\bar K_n^u(x')\right]}.\label{3.73a}
\eea
For $\bar G_n^{u} \equiv\PP_{\pi_n}\left(\max_{i\in [\theta_n]} e^{\sqrt{n}\b_n H^0(i)}e_{n,i} \leq c_n u^{1/\alpha_n}\right)$ observe that
\be\label{3.73b}
\eqref{3.73}\leq k_n(t)- 2 k_n(t) \E\bar G_n^{u}.
\ee
We add and subtract $\E \EE_{\pi_n} (1- K_n^u)\equiv\E \EE_{\pi_n}\bar K_n^u$ as well as  
\be\label{3.74}
\textstyle{
\sum_{(x,x')\in \S_n^2}\mu_n(x,x')\E \EE_{x} \bar K_n^u(x) \EE_{x'}\bar K_n^u(x')}.
\ee
Re-arranging the terms and using the bound from \eqref{3.73b} we see that $\E\eta_n^t(u)$ is bounded from above by
\bea\label{3.75}
&&2k_n(t)\left(\E \bar K_n^u - \E\bar G_n^{u} \right)\\
&+&k_n(t)\sum_{x,x'}\mu_n(x,x') \E \EE_{x}K_n^u(x) \E \EE_{x'}K_n^u(x')\label{3.76}\\
&+&k_n(t)\sum_{x,x'}\mu_n(x,x')\left(\E \left[\EE_{x}\bar K_n^u(x) \EE_{x'}\bar K_n^u(x')\right] - \E \EE_{x}\bar K_n^u(x) \E \EE_{x'}\bar K_n^u(x')\right). \ \ \ 
\label{3.77}
\eea
From Proposition \ref{prop2.2} we conclude that \eqref{3.75} and \eqref{3.76} are of order $O\left(\frac{\log n}{n}\right)$ and $O\left(\theta_n a_n^{-1}\right)$ respectively. To control \eqref{3.77} we use the normal comparison theorem (Lemma \ref{lemma2.5}) for the processes $V^0$ and $V^1$ as in Proposition \ref{prop2.8}. However, due to the fact that we are looking at the chain after $\bar \theta_n$ steps, the comparison is simplified. More precisely, let $\cal{A}_n\equiv \left\{\forall \bar \theta_n\leq i\leq \theta_n : \ \dist(J_n(i),J'_n(i))>n(1-\rho(n))\right\}$ $\subset \cal{F}^{J}\times\cal{F}^{J'}$, where $\rho(n)$ is of the order of $\sqrt{n^{-1} \log n}$. Then, on $\cal{A}_n$, by Lemma \ref{lemma2.5} and the estimates from \eqref{3.34},
\be\label{3.78}
\E \left[ \bar K_n^u(x) \bar K_n^u(x')\right] - \E \bar K_n^u(x) \E \bar K_n^u(x') \leq 2 \gamma_n^{-2}
\sum_{{1 \leq i\leq\theta_n }\atop{ \theta_n+1 \leq j\leq2\theta_n}} \Delta_{ij}^0 e^{-\gamma_n^2 n (1+\Delta^0_{ij})^{-1}}\leq O(\theta_n^{2}a_n^{-2}).
\ee
Moreover, on $\cal{A}_n^c$,
\be\label{3.79}
\E \left[ \bar K_n^u(x) \bar K_n^u(x')\right] - \E \bar K_n^u(x) \E \bar K_n^u(x') \leq O(a_n^{-1}).
\ee
But in Lemma 3.7 from \cite{BG10} it is shown that for a specific choice of $\rho(n)$ and every $x \in \S_n$
\bea\label{3.80}
P\left(\cal{A}_n | \dist(J_n(0),J'_n(0))=2\right) &\geq& 1-n^{-8}\nonumber\\
P_x\left(\cal{A}_n^c\right) &\leq& n^{-4}.
\eea
Therefore we obtain that $\lim_{n\rightarrow \infty} \E \eta_n^t(u) = 0$.
\end{proof}
\begin{remark}
Lemma \ref{lemma3.10} immediately implies that the second part of Condition (2-1) holds in $\P$-probability. To show that it is satisfied $\P$-almost surely for $p>5$ and $c\in(0,\frac 12)$ or $p=5$ and $c<\frac14$ it suffices to control the variance of \eqref{3.75}. We use the same concentration results as in Proposition \ref{prop2.8} to obtain that the variance of $k_n(t)(\bar K_n^u - \bar G_n^{u})$, which is given by
\be
k_n^2(t)\left[\E\left(\bar K_n^u-\E \bar K_n^u\right)^2+\E \left(\bar G_n^{u}-\E\bar G_n^{u}\right)^2-2\left(\E \bar G_n^{u} \bar K_n^u-\E \bar G_n^{u} \E\bar K_n^u\right)\right],
\ee
is bounded from above by $C\gamma_n^{-2} n^{1-p/2}$.

\end{remark}

\subsection{Condition (3-1)} \label{S34}
We show that Condition (3-1) is $\P$- a.s. satisfied for all $\delta>0$.
\begin{lemma}\label{lemma2.11}
We have $\P$-a.s. that
\be\label{3.82}
\limsup_{n\rightarrow \infty}\left(a_n \left(c_n \delta^{1/\alpha_n} \right)^{-1}\EE_{\pi_n} \lambda^{-1}_n(J_n(1)) e_{n,1} \1_{\lambda^{-1}_n(J_n(1)) e_{n,1} \leq c_n \delta^{1/\alpha_n}}\right)^{\alpha_n}< \infty, \quad \forall \delta>0.
\ee
\end{lemma}
\begin{proof}
We begin by proving that for all $\delta>0$, for $n$ large enough,
\bea\label{3.83}
\textstyle{\frac{a_n}{c_n \delta^{1/\alpha_n}} \EE_{\pi_n} \E \lambda^{-1}_n(J_n(1)) e_{n,1} \1_{\lambda^{-1}_n(J_n(1)) e_{n,1} \leq c_n \delta^{1/\alpha_n}}}&=& \textstyle{\sum_{x\in \S_n} 2^{-n}\E Y_{n,\delta}(x)}\nonumber\\ 
&\leq& 4 (\delta \gamma_n \beta_n)^{-1},
\eea
where $Y_{n,\delta}(x)\equiv a_n \left(c_n \delta^{1/\alpha_n} \right)^{-1} \lambda^{-1}_n(x) e_{n,1}  \1_{\lambda^{-1}_n(x) e_{n,1} \leq c_n \delta^{1/\alpha_n}}$, for $x \in \S_n$.

For $x\in \S_n$ we have that
\bea\label{3.84}
\E Y_{n,\delta}(x)&=& a_n(c_n \delta^{1/\alpha_n})^{-1} (2\pi)^{-1/2} \int_0^{\infty}dy \int_{-\infty}^{y_n}dz\ y e^{-y-\frac{z^2}{2}+\b_n \sqrt{n}z} \nonumber\\
&=&a_n(c_n \delta^{1/\alpha_n})^{-1} (2\pi)^{-1/2}\int_0^{\infty}dy \int_{\b_n \sqrt{n}-y_n}^{\infty} dz\ y e^{-y+\frac{\b_n^2 n}{2}-\frac{z^2}{2}} ,
\eea
where $y_n\equiv (\sqrt{n}\b_n)^{-1} \left(\log c_n + \frac{\beta_n}{\gamma_n}\log \delta - \log y\right)$ for $y>0$.
In order to use estimates on Gaussian integrals, we divide the integration area over $y$ into $y \leq n^2$ and $y>n^2$.

For $y>n^2$, there exists a constant $C'>0$ such that
\be\label{3.87}
(2\pi)^{-1/2}a_n (c_n\delta^{1/\alpha_n})^{-1} \int_{n^2}^{\infty}dy \int_{-\infty}^{y_n}dz\ y e^{-y-\frac{z^2}{2}+\b_n \sqrt{n}z} \leq C' a_n  n^4 e^{-n^2},
\ee
which vanishes as $n \rightarrow \infty$.

Let $y\leq n^2$. By definition of $c_n$ we have $\b_n \sqrt{n}-y_n  = \sqrt{n} \b_n \left(1-\tfrac{\gamma_n}{\beta_n}-\tfrac{\log \delta}{\gamma_n\beta_n n}+\tfrac{\log y}{\beta^2_n n}\right)$. Since $\alpha_n\downarrow 0$ as $n \rightarrow \infty$, it follows that for $n$ large enough $\b_n\sqrt{n}-y_n> 0$. 
But then, since $\P(Z>z)\leq (\sqrt{2\pi})^{-1} z^{-1} e^{-z^2/2}$ for any $z>0$ and $Z$ being a standard Gaussian,
\be\label{3.85}
\int_0^{n^2}dy \int_{-y_n+\b_n \sqrt{n}}^{\infty} dz\ y e^{-y+\frac{\b_n^2 n}{2}-\frac{z^2}{2}} \leq \int_0^{n^2}dy \frac{y e^{-y}}{\b_n\sqrt{n}-y_n}e^{\frac{\b_n^2 n}{2}-\frac{(\b_n\sqrt{n}-y_n)^2}{2}}\ .
\ee
Plugging in the definition of $a_n$ and $c_n$, \eqref{3.87} and \eqref{3.85} yield that, for $n$ large enough, up to a multiplicative error that tends to $1$ as $n\rightarrow\infty$ exponentially fast,
\bea\label{3.86}
\eqref{3.84}&\leq& \textstyle{\int_0^{n^2} dy\ y^{\alpha_n} e^{-y}(\gamma_n\beta_n\delta)^{-1} \left(1-\frac{\gamma_n}{\beta_n} -\frac{\log \delta}{n\gamma_n\beta_n}+
\frac{\log y}{\beta_n^2 n} \right)^{-1} e^{2 \log \delta \log n (n\gamma_n\beta_n)^{-1}}}\nonumber\\
&\leq& 2 \int_0^{n^2} dy\ y^{\alpha_n} e^{-y}(\gamma_n\beta_n\delta)^{-1} \nonumber\\
&\leq& \textstyle{2 \Gamma\left(1+\frac{\gamma_n}{\beta_n}\right) \left(\gamma_n\beta_n\delta\right)^{-1},}
\eea
where $\Gamma(\cdot)$ denotes the gamma function. Since $\Gamma(1+\alpha_n)\leq 1$ for $\alpha_n\leq 1$, the claim of \eqref{3.83} holds true  for all $\delta>0$ for $n$ large enough.

Lemma 3.10 from \cite{BG10} yields that for all $\delta>0$ there exists $\kappa>0$ such that
\be\label{3.88}
\E \left(\EE_{\pi_n} Y_{n,\delta}\right)^2 - \left( \E\EE_{\pi_n} Y_{n,\delta}\right)^2 \leq a_n^2 \left(c_n \delta^{1/\alpha_n} \right)^{-2}
 n^{1-p/2}\leq e^{-n^{\kappa}},
\ee
where $\EE_{\pi_n}Y_{n,\delta}\equiv \sum_{x \in \S_n} 2^{-n}Y_{n,\delta}(x)$. For all $\delta>0$  there exists by Borel-Cantelli Lemma a set $\Omega(\delta)$ with $\P(\Omega(\delta))=1$ such 
that on $\Omega(\delta)$, for all $\varepsilon>0$ there exists $n' \in \N$ such that
\be\label{3.89}
\EE_{\pi_n} Y_{n,\delta}\leq 4 \left(\gamma_n\beta_n\delta\right)^{-1} + \varepsilon , \quad \forall n\geq n' .
\ee
Setting $\Omega^{\tau}\equiv \bigcap_{\delta\in \Q\cap(0,\infty))} \Omega(\delta)$, we have $\P(\Omega^{\tau})=1$.

Let $\delta>0$ and $\varepsilon>0$. We can always find $\delta' \in \Q$ such that $\delta\leq \delta'\leq 2\delta$. Note that $Y_{n,\delta}$
 is increasing in $\delta$. Moreover, by \eqref{3.89} there exists $n'=n'(\delta', \varepsilon)$ such that on $\Omega^{\tau}$ and for $n\geq n'$
\be\label{3.90} 
\left(\EE_{\pi_n} Y_{n,\delta} \right)^{\alpha_n} \leq \left(\EE_{\pi_n}Y_{n,\delta'}\right)^{\alpha_n}
 \leq \left(4 \left(\gamma_n\beta_n\delta'\right)^{-1} + \varepsilon\right)^{\alpha_n} \leq 
4\left(\gamma_n\beta_n\delta'\right)^{-\alpha_n}.
\ee
Since $(\gamma_n\beta_n)^{-\alpha_n}\downarrow 1$ as $n \rightarrow \infty$, we obtain the assertion of Lemma \ref{lemma2.11}.
\end{proof}

\subsection{Proof of Theorem \ref{p:main}} \label{S35}
We are now ready to conclude the proof of Theorem \ref{p:main}.

First let $p> 5$ and $\gamma_n=n^{-c}$ for $c\in \left(0,\frac 1 2\right)$, or $p=5$ and $c>\frac 1 4$. Then we know by Propositions \ref{prop2.2} and \ref{prop2.8} that for all $u>0$ there exists a set $\Omega(u)$ with $\P (\Omega(u))=1$, such that on $\Omega(u)$
\be\label{3.91}
\lim_{n\rightarrow\infty} \nu_n^t(u,\infty)=K_p t u^{-1}, \quad \forall t>0.
\ee
The mapping that maps $u$ to $\nu_n^t(u,\infty)$ is decreasing on $\left(0, \infty\right)$ and its limit, $u^{-1}$, is continuous on the same interval. Therefore, setting $\Omega_1^{\tau}=\bigcap_{u \in \left(0,\infty\right)\cap\Q} \Omega(u)$, we have $\P(\Omega_1^{\tau})=1$ and \eqref{3.91} holds true for all $u>0$ on $\Omega_1^{\tau}$. By the same arguments and the results in Section \ref{S33} there also exists a subset $\Omega_2^{\tau}$ with full measure and such that the second part of Condition (2-1) holds on $\Omega_2^{\tau}$.

Condition (3-1) holds $\P$-a.s. by Lemma \ref{lemma2.11}. Finally, we are left with the verification of Condition (0) for the invariant measure $\pi_n(x)=2^{-n}$, $x\in \S_n$. For $v>0$, we have that
\be\label{3.92}
\sum_{x\in\S_n}2^{-n}e^{-v^{\alpha_n}c_n \lambda_n(x)}= \sum_{x\in\S_n} 2^{-n} \PP_{\pi_n}\left(\lambda^{-1}_n(x) e_{n,1}> c_n v^{\alpha_n}\right) .
\ee
By similar calculations as in \eqref{3.86}, we see that, for $n$ large enough and $x\in\S_n$,
\be\label{3.93}
\E \PP_{\pi_n}\left(\lambda^{-1}_n(x) e_{n,1}> c_n v^{\alpha_n}\right) \sim a_n^{-1} \gamma_n^2 v^{-1},
\ee
which tends to zero as $n \rightarrow \infty$. By a first order Chebychev inequality we conclude that for all $v>0$ Condition (0) is satisfied $\P$-a.s. As before, by monotonicity and continuity, this implies that Condition (0) holds $\P$-a.s. for all $v>0$. This proves Theorem \ref{p:main} in this case.

For $p=2,3,4$ and $c\in \left(0,\frac 1 2\right)$ or $p=5$ and $c\geq \frac 1 4$, we know from Propositions \ref{prop2.2}, \ref{prop2.8}, and Section \ref{S33} that Condition (2-1) is satisfied in $\P$-probability, whereas Condition (0) and (3-1) hold $\P$-a.s. This concludes the proof of Theorem \ref{p:main}. 

\subsection{Proof of Theorem \ref{p:corr}} \label{S36}
We use Theorem \ref{p:main} to prove the claim of Theorem \ref{p:corr}. By the same arguments as in the proof of Theorem 1.5 in \cite{BG10}, we obtain that for $t>0$, $s>0$, and $\varepsilon\in(0,1)$ the correlation function $\CC_n^{\varepsilon}(t,s)$ can, with very high probability and $\P$- a.s., be approximated by
\bea\label{3.94}
\CC_n^{\varepsilon}(t,s)&=&(1-o(1))\ \PP_{\pi_n}(\cal{R}_n \cap (t^{\alpha_n}, (t+s)^{\alpha_n}) =\emptyset) \nonumber\\
&=&(1-o(1))\ \PP_{\pi_n}(\cal{R}_{\alpha_n} \cap (t, t+s) =\emptyset),
\eea
where $\cal{R}_n$ is the range of the blocked clock process $S_n^{b}$ and $\cal{R}_{\alpha_n}$ is the range of $\left(S_n^{b}\right)^{\alpha_n}$. By Theorem \ref{p:main} we know that $\left(S_n^{b}\right)^{\alpha_n}\limj M_{\nu}$, $\P$-a.s. for $p>5 $ if $c\in(0,\frac 12)$, $p=5$ if $c<\frac 14$, and in $\P$-probability else. By Proposition 4.8 in \cite{Re} we know that the range of $M_{\nu}$ is the range of a Poisson point process $\xi'$ with intensity measure $\nu'(u,\infty)=\log u-\log K_p$. Thus, writing $\cal{R}_M$ for the range of $M_{\nu}$, we get that
\be\label{3.95}
\textstyle{\PP (\cal{R}_M \cap (t,t+s)=\emptyset)=\PP(\xi'(t,t+s)=0)= e^{-\nu'(t,t+s)}=\frac{t}{t+s}.}
\ee
The claim of Theorem \ref{p:corr} follows.
\section{Appendix} \label{S4}
In the appendix we state and prove a lemma that is needed in the proof of Lemma \ref{lemma2.6}.
\begin{lemma}\label{lemma2.7}
Let $D_{ij}=\dist(J_n(i),J_n(j))$ and  $\Delta^0_{d}=(1-2dn^{-1})^p$. For any $\eta>0$ there exists a constant $\bar C<\infty$ such that, for $n$ large enough and $d\in \{0,\ldots,n\}$,
\bea
&&\textstyle{
 k_n(t) \sum_{\lf i/v_n\rf =\lf j/v_n\rf}^{\theta_n}E_{\pi_n}\1_{D_{ij}=d} |\Delta^0_d-\Delta^1_{ij}|  }\leq \bar C t a_n \frac{d^2}{v_n n} \1_{d\leq v_n}\label{3.43},\\
&&\textstyle{
k_n(t) \sum_{\lf i/v_n\rf \neq\lf j/v_n\rf}^{\theta_n}E_{\pi_n}\1_{D_{ij}=d}  }\leq \bar C t \frac{a_n \exp\left(\eta \gamma_n^2 \min\{d,n-d\}\right)}{v_n \gamma_n^2}\label{3.44}.
\eea
\end{lemma}
\begin{proof} We use ideas from Section 3 in \cite{BBC08} and Section 4 in \cite{BAGun11} and write the distance process $D_{ij}=\dist(J_n(i),J_n(j))$ as the Ehrenfest chain $Q_n = \{Q_n(k):\ k\in \N\}$, which is a birth-death process with state space $\{0,\ldots, n\}$ and transition probabilities $p_{k,k-1}=1-p_{k,k+1} = \frac k n$ for $k\in \{0,\ldots,n\}$. Denote by $P_{k}$ the law and $E_k$ the expectation of $Q_n$ starting in $k$. Let moreover $T_d=\inf\{k\in\N: \ Q_n(k)=d\}$. By the Markov property of $J_n$, we have under $P_0$, in distribution, that
\be\label{4.1}
\dist(J_n(0),J_n(k))\stackrel{d}{=} \dist(J_n(j),J_n(j+k))\stackrel{d}{=}Q_n(k) \ , \quad \forall j,k\geq 0.
\ee
Recall for the proof of \eqref{3.43} that if $\lf i/v_n\rf =\lf j/v_n\rf$, we have that $\Delta_{ij}^1 \leq \Delta_{i,j}^0$. Moreover, since for such $i,j$ necessarily $|i-j|\leq v_n$ we have that $D_{ij}\leq v_n$. 
Thus, let $d\in\{1,\ldots,v_n\}$. By Lemma 4.2 in \cite{BBC08} we deduce that there exists a constant $C<\infty$, independent of $d$, such that
\be\label{4.3}
\textstyle{ k_n(t) \sum_{\lf i/v_n\rf =\lf j/v_n\rf}^{\theta_n}E_{\pi_n}\1_{D_{ij}=d} \leq C t a_n} \ .
\ee
Moreover,
\be\label{4.4}
\textstyle{
\left(\Delta_d^0-\Delta_{ij}^1\right)=\left(1- \frac{2d}{n}\right)^p-\left(1-\frac{2p|i-j|}{n}\right) = \frac{2p}{n}\left(|i-j|-d\right) + O\left(\frac{d^2}{n^2}\right).}
\ee
Therefore the main contributions in \eqref{3.43} are of the form
\bea\label{4.5}
\textstyle{\sum_{\lf i/v_n\rf =\lf j/v_n\rf}^{\theta_n}\left(|i-j|-d\right) E_{\pi_n}\1_{D_{ij}=d} }&=&\textstyle{v_n \sum_{i=1}^{\lf \theta_n/v_n\rf}\sum_{j=i+1}^{i+v_n}\left(j-i-d\right) E_{\pi_n}\1_{D_{ij}=d}}\nonumber\\
&=&\textstyle{v_n \sum_{i=1}^{\lf \theta_n/v_n\rf}\sum_{j=1}^{v_n} E_{0}\1_{Q_n(j)=d}\left(j-d\right)}.
\eea
Setting $Z\equiv \sum_{j=1}^{v_n}\1_{Q_n(j)=d}\left(j-d\right)$, \eqref{4.5} is nothing but $\theta_n E_0 Z$. It is shown in \cite{BAGun11} (page 107-108) that there exists a constant $C<\infty$, independent of $d$, such that
\bea \label{4.6}
E_0 Z &\leq& C E_0 \left(T_d - d\right)\1_{T_d<v_n} \nonumber\\
&\leq& C \left(E_0 T_d  - d P_0\left(T_d<v_n\right)\right) \leq C \left(E_0 T_d  - d\left(1- v_n^{-1} E_0 T_d\right)\right)\ ,
\eea
where the last inequality is obtained by a first order Chebychev inequality.
To calculate $E_0 T_d$ we use the following classical formulas (see e.g. \cite{LPW09}, Chapter 2.5)
\bea \label{4.7}
E_0 T_d &=& \textstyle{\sum_{l=1}^d E_{l-1} T_l} , \quad \mbox{ where }\\
E_{l-1} T_l &=& \textstyle{ \frac 1 {p_{l,l-1}} \prod_{i=1}^l \frac{p_{i,i-1}}{ p_{i-1,i}}\left(1+ \sum_{j=1}^{l-1} \prod_{k=1}^j  \frac{p_{k,k-1}}{p_{k-1,k}}\right)}.  \label{4.8}
\eea
Plugging in the transition probabilities, we obtain for all $l\leq d$,
\bea \label{4.9}
E_{l-1} T_l &=& \textstyle{ \frac n l \left(\prod_{i=1}^l \frac i {n-i+1}+ \sum_{j=1}^{l-1} \prod_{k=j+1}^l  \frac k{n-k+1}\right)}\nonumber\\
&=& \textstyle{ \frac n l \sum_{j=0}^{l-1} \prod_{k=j+1}^l \frac k {n-k+1}}.
\eea
For any $l\leq d$ and $0\leq j\leq l-1$ we have that
\be\label{4.10}
\textstyle{
\frac n l \prod_{k=j+1}^l \frac k {n-k+1} \leq  \frac n d \prod_{k=j+1}^l \frac d {n-d} \ .}
\ee
In view of \eqref{4.7} we get that
\be\label{4.11}
\textstyle{
E_0 T_d \leq \sum_{l=1}^{d}  \frac 1 {1-2dn^{-1}} \left(1-  \left(\frac{d}{n-d}\right)^l \right) \leq \frac d {(1-2dn^{-1})}  }.
\ee
But then, since $\frac d n \downarrow 0$ as $n \rightarrow \infty$ and $d\leq v_n$, there exists a constant $C'<\infty$, independent of $d$, such that
\be\label{4.12}
\textstyle{
E_0 Z \leq C' \frac{d^2}{v_n}.}
\ee
Together with \eqref{4.3} and \eqref{4.4} this concludes the proof of \eqref{3.43}.

For the proof of \eqref{3.44} we distinguish several cases. 
If $\|d\|\equiv \min\{d, n-d\} > (\log n)^{1+\varepsilon} \gamma_n^{-2}$ for some fixed $\varepsilon>0$ then the claim of \eqref{3.44} is deduced from the bound
\be\label{4.13}
\textstyle{k_n(t) \sum_{\lf i/v_n\rf \neq\lf j/v_n\rf}^{\theta_n}E_{\pi_n}\1_{D_{ij}=d}  \leq a_n t \theta_n \ll a_n t \frac{e^{\eta \|d\|  \gamma_n^2}}{v_n \gamma_n^2}} .
\ee
Assume next that $\|d\| \leq  (\log n)^{1+\varepsilon} \gamma_n^{-2}$. It is shown in \cite{BAGun11}, (page 111-112), that in this case one can neglect values of $d$ such that $d\geq \frac n 2$. Thus, let $d \leq  (\log n)^{1+\varepsilon} \gamma_n^{-2}$. Note that
\be\label{4.14}
\textstyle{k_n(t) \sum_{\lf i/v_n\rf \neq\lf j/v_n\rf}^{\theta_n}E_{\pi_n}\1_{D_{ij}=d}  \leq k_n(t) \sum_{k=0}^{\theta_n}\sum_{m=j_k}^{\theta_n} E_{\pi_n}\1_{D_{k,k+m}=d} \ ,}
\ee
where $j_k=\inf\{i \in \N: \ \lf k/v_n\rf \neq\lf (k+i)/v_n\rf \}$. 

We further distinguish the cases $j_k\leq 2d$ and $j_k>2d$. If $j_k \leq 2d$ then, setting $Z_{j_k}(d)\equiv \sum_{m=j_k}^{\theta_n}\1_{D_{k,k+m}=d}$, we have $Z_{j_k}(d)\leq Z_0(d)$. It is shown on page 685 in \cite{BBC08} that there exists $C<\infty$, independent of $d$, such that $E_0 Z_0(d)\leq C$. Since moreover $|\{k\in \{1,\ldots,\theta_n\} : \ j_k \leq 2d\}|\leq 2 \frac{d \theta_n}{v_n}$, we know that for all $\eta>0$ there exists $C'<\infty$ such that
\be \label{4.15}
\textstyle{k_n(t) \sum_{k=0}^{\theta_n}\sum_{m=j_k}^{\theta_n} E_{\pi_n}\1_{D_{ij}=d} \leq Ct \frac{a_n d}{ v_n} \leq C't \frac{a_n e^{\eta \gamma_n^2 \|d\|}}{ v_n \gamma_n^{2}} .}
\ee
Let $j_k > 2d$, i.e. in particular $ Z_{j_k}(d)\leq Z_{2d}(d)$. By the Markov property and by Lemma 4.2 in \cite{BBC08} we obtain that there exists $C<\infty$ such that
\be \label{4.16}
\textstyle{E_0 Z_{2d}(d) \leq P_0 (T_d \in (2d,\theta_n))\left(1+ E_d\left(\sum_{k=1}^{\theta_n} \1_{Q_n(k)=d}\right)\right) \leq C P_0(T_d \in (2d,\theta_n)).}
\ee
The probability that $Q$ gets from $0$ to $d$ after $2d$ steps is bounded by the probability that it takes at least $d$ steps to the left, i.e. 
\be\label{4.17}
\textstyle{
P_0(T_d \in (2d,\theta_n))\leq  {2d\choose d}\left(\frac d n\right)^{d}\leq 2d \left(\frac{4d} n\right)^d \ll \frac d{v_n }.}
\ee
The claim follows as in \eqref{4.15}. This finishes the proof of \eqref{3.44}.
\end{proof}
\bibliographystyle         {abbrv}      

\def\soft#1{\leavevmode\setbox0=\hbox{h}\dimen7=\ht0\advance \dimen7
  by-1ex\relax\if t#1\relax\rlap{\raise.6\dimen7
  \hbox{\kern.3ex\char'47}}#1\relax\else\if T#1\relax
  \rlap{\raise.5\dimen7\hbox{\kern1.3ex\char'47}}#1\relax \else\if
  d#1\relax\rlap{\raise.5\dimen7\hbox{\kern.9ex \char'47}}#1\relax\else\if
  D#1\relax\rlap{\raise.5\dimen7 \hbox{\kern1.4ex\char'47}}#1\relax\else\if
  l#1\relax \rlap{\raise.5\dimen7\hbox{\kern.4ex\char'47}}#1\relax \else\if
  L#1\relax\rlap{\raise.5\dimen7\hbox{\kern.7ex
  \char'47}}#1\relax\else\message{accent \string\soft \space #1 not
  defined!}#1\relax\fi\fi\fi\fi\fi\fi} \def\cprime{$'$} \def\cprime{$'$}
  \def\cprime{$'$}

 \end{document}